\documentclass[10pt,a4paper]{article}
\usepackage[utf8]{inputenc}
\usepackage[T1]{fontenc}
\usepackage{amsmath}
\usepackage{amsfonts}
\usepackage{amssymb}
\usepackage{amsthm}
\usepackage{xcolor}
\usepackage{bbm}
\usepackage[left=2cm,right=2cm,top=2cm,bottom=2cm]{geometry}

\newtheorem{thm}{THEOREM}[section]
\newtheorem{assumption}[thm]{Assumption}
\newtheorem{cor}[thm]{COROLLARY}

\newtheorem{definition}[thm]{DEFINITION}
\newtheorem{example}[thm]{EXAMPLE}
\newtheorem{proposition}[thm]{PROPOSITION}
\newtheorem{remark}[thm]{REMARK}

\providecommand{\norm}[1]{\left\lVert#1\right\rVert}
\providecommand{\abs}[1]{\left\lvert#1\right\rvert}
\providecommand{\pr}[1]{\left(#1\right)} 
\providecommand{\pp}[1]{\left[#1\right]} 
\providecommand{\set}[1]{\left\lbrace#1\right\rbrace} 
\providecommand{\scal}[1]{\left\langle#1\right\rangle}

\providecommand{\keywords}[1]{\textbf{\textit{Keywords:  }} #1}
\providecommand{\classification}[1]{\textbf{\textit{AMS MSC: }} #1}

\begin{document}
\author{Dan GOREAC\footnote{Université Paris-EstUniversité Paris-Est, LAMA (UMR 8050), UPEMLV, UPEC, CNRS, F-77454, Marne-la-Vallée, France, Corresponding author, email:  dan.goreac@u-pem.fr}}
\title{Approximately Reachable Directions for Piecewise Linear Switched Systems}
\date{}
\maketitle
\begin{abstract}{This paper deals with some reachability issues for piecewise linear switched systems with time-dependent coefficients and multiplicative noise. Namely, it aims at characterizing data that are almost reachable at some fixed time $T>0$ (belong to the closure of the reachable set in a suitable $\mathbb{L}^2$-sense). From a mathematical point of view, this provides the missing link between approximate controllability towards $0$ and approximate controllability towards given targets. The methods rely on linear-quadratic control and Riccati equations.  The main novelty is that we consider an LQ problem with \textit{controlled backward stochastic dynamics} and, since the coefficients are \textit{not deterministic} (unlike some of the cited references), neither is the backward stochastic Riccati equation. Existence and uniqueness of the solution of such equations rely on structure arguments (inspired by \cite{CFJ_2014}). Besides solvability, Riccati representation of the resulting control problem is provided as is the synthesis of optimal (non-Markovian) control. Several examples are discussed.}\end{abstract}
\classification{93B05, 93B25, 60J75}\\
\keywords{Reachability; Approximate controllability; Controlled switch process; Linear-quadratic control; Backward stochastic Riccati equation; Stochastic gene networks}
\section{Introduction}
We consider systems consisting of two components (denoted for simplicity $\pr{\Gamma_t,X_t}$) dynamically evolving as the time $t\in\pp{0,T}$ changes up to some finite time horizon $T>0$. The first component switches at times $T_j$ between a family of modes living in some space $E$. The precise description will be given in the following section. As usual, a marked point measure $q$ is associated to $\Gamma$. Its compensator will be denoted by $\hat{q}$ and the martingale measure by $\tilde{q}:=q-\hat{q}$. As the mode switches, the component $X$ living in some Euclidean space $\mathbb{R}^n$ and obeying a linear controlled equation governed by predictable coefficient matrices and presenting multiplicative noise \[dX_t^u=\pr{A_tX_t^u+B_tu_t}dt+\int_EC_t(\theta)X_t^u\tilde{q}(dtd\theta),\textnormal{ for }t\in\pp{0,T},\]
changes its flow. (The simplest model one should have in mind is the case when $A,B,C$ keep track of the modes encountered up to time $t$). Such systems are best described using tools in point processes (e.g. \cite{Ikeda_Watanabe_1981}, \cite{Bremaud_1981}), but they also intersect particular cases of piecewise deterministic Markov processes (introduced in \cite{Davis_84}, \cite{davis_93}). They intervene, for instance, in mathematical models of biochemical reactions (describing gene networks). In this case, the exogenous control parameter ($u$) is a means to alter (enhance or inhibit) desired properties. Two kinds of (closely related) problems are very natural in this framework. The first (in a progressive perspective) strives to provide an answer to the question of being able to drive the system to desired outcomes (say targets $\xi$). The second, obeying to a passéist perspective, aims at providing a model (structure) validation by looking at effective data (still $\xi$) at some point in time and checking it against feasible (reachable) outcomes.\\
The ability to drive (in convenient manner by altering the control) the dynamical system to a favorable outcome is commonly known as \textit{controllability}. The subject has generated a consistent literature starting from the pioneer work \cite{Kalman_1959}. For non-random systems, one usually aims at \textit{exact controllability} (precisly reaching the target) and the tools are either algebraic or rely on estimates. In this setting, direct approach relies on Riccati techniques (Grammian matrix in the simplest framework) and dual methods on the notion of \textit{observability} (e.g. \cite{Hautus}) translating into invariance properties of convenient linear (sub)spaces or \textit{admissibility} criteria. Either method has been extensively generalized to infinite-dimensional systems (e.g. \cite{Schmidt_Stern_80}, \cite{Curtain_86}, \cite{russell_Weiss_1994}, \cite{Jacob_Partington_2006}, etc.).\\
For random systems, duality techniques rely on backward stochastic differential equations (BSDE introduced in \cite{Bismut_73} and generalized to nonlinear cases in \cite{Pardoux_Peng_90})\footnote{The current literature on BSDE is highly rich but, for our linear context, it seems superfluous to mention recent advances on discontinuity, terminal irregularity, super-linear coefficient growth, etc.}. This tool allows to give a final answer to the notion of exact controllability (cf. \cite{Peng_94}). Some refinement has been recently obtained in Brownian setting by relaxing the class of controls in \cite{LQiYongJiongminZhangXu_2012} and \cite{WangYangYongYu2016}. These papers make use of some full-rank condition(s). When such conditions fail to hold (e.g. $B$ is not of full rank and the system has no control on the noise part as it is the case here), one has to settle for weaker controllability properties. Controlling stochastic systems arbitrarily close to the target (\textit{approximate controllability}) in a multiplicative noise setting  fails to be captured by the (deterministic-like) drift. These properties have been addressed in \cite{Buckdahn_Quincampoix_Tessitore_2006} and \cite{G17} for finite-dimenional Brownian diffusions with constant coefficients. Finally, let us mention that various methods lead to partial results on various types of controllability in infinite dimensions (e.g. \cite{Fernandez_Cara_Garrido_atienza_99}, \cite{Sarbu_Tessitore_2001}, \cite{Barbu_Rascanu_Tessitore_2003}, \cite{LQiYongJiongminZhangXu_2012}, etc.).
In the recent papers \cite{GoreacMartinez2015} and \cite{GoreacGrosuRotenstein_2016} we have addressed the approximate controllability problem with switched dynamics (much like those presented at the very beginning). In particular, \cite{GoreacGrosuRotenstein_2016} gives a complete criterion allowing to control the system around $0$ (\textit{approximate null-controllability}). However, whenever the coefficient are not constant, approximate null-controllability fails to imply the ability to control the system to arbitrary targets (as it was the case in \cite{Buckdahn_Quincampoix_Tessitore_2006} and \cite{G17} treating time-homogeneous settings).  This paper aims at providing  the missing link between approximate null and approximate (full) controllability. \\
Let us briefly explain the general intuitions. In \cite{Peng_94}, the key ingredient in addressing the (exact) controllability problem is the so-called (exact) terminal controllability i.e. the ability to solve some BSDE with terminal datum $\xi$. Of course (again due to the general theory of BSDEs), solving 
\[X_T^u=\xi,\ dX_t^u=\pr{A_tX_t^u+B_tu_t}dt+\int_EC_t(\theta)X_t^u\tilde{q}(dtd\theta),\textnormal{ for }t\in\pp{0,T},\] cannot be achieved for every square-integrable $\xi$ (it is clear that this cannot be done for every $u$ and, often, it cannot be done for any control). However, for every (suitably measurable and integrable) $u$, one is able to solve 
\[X_T^u=\xi,\ dX_t^u=\pr{A_tX_t^u+B_tu_t}dt+\int_E\pp{C_t(\theta)X_t^u+\mathbf{Z_t^u(\theta)}}\tilde{q}(dtd\theta),\textnormal{ for }t\in\pp{0,T}.\] Then, $\xi$ is approximatly reachable (or, equivalently, the system $X$ is approximately controllable to $\xi$) if, by varying the control $u$, the noise component $Z^u(\theta)$ is close (in an $\mathbb{L}^2$ sense) to $0$. In other words, to characterize approximately reachable targets $\xi$, one looks at the level set $\set{\xi:V(T,\xi)=0}$, where \[V(T,\xi):=\inf_{u}\mathbb{E}\pp{\int_0^T\norm{Z_t^u(\theta)}^2\hat{q}(dtd\theta)}.\]Again, the class of controls over which such infimum is taken is left to be specified latter on.
We are now in the presence of a quadratic control problem with linear dynamics (thus an LQ-problem) with the small reserve that the dynamics are BSDEs (instead of the more common forward control systems). This kind of problem is not entirely new. To our best knowledge, it has been considered for the first time with Brownian dynamics in \cite{LimZhou2001} (see also some extensions to mean-field systems in \cite{LiSunXiong2017}). As usual, LQ problems are best addressed by exhibiting a convenient class of Riccati equations. We will also follow this long-established program. However (besides the difference in stochasticity), the cited papers deal with deterministic driving coefficients. As a result, their resulting Riccati equations are deterministic. This will no longer be the case for us. The Riccati system quantifying $V(T,\xi)$ is a true backward stochastic Riccati equation. A class of backward-type stochastic Riccati equations for systems governed by a marked point mechanism has been recently studied in \cite{CFGT_2018}. Although the stochastic framework is more general in \cite{CFGT_2018} (both for the compensator intervening and the presence of a Brownian term), the generator does not include the second part of the solution and does not fit the present problem. We have chosen an approach based on structural properties (cf. \cite{CFJ_2014}) and allowing to reduce the BSDE to a class of ordinary equations (of Riccati type in our case). The resulting equivalent system (\ref{RiccatiDet}) is composed of (iterated) standard deterministic Riccati equations for which qualitative properties (existence, uniqueness, positiveness, lower and upper-bounds, monotony) are investigated.\\
We begin with a description of the stochastic model, notations and definitions in Section \ref{SectionPreliminaries}. Section \ref{SectionBasicExp} is dedicated to an abstract (operator) approach to the notion of approximate terminal-controllability and some examples. In general, the notions of (approximate null and full) controllability are meant as \textit{global} (i.e. required with respect to \textit{every fixed initial datum}, see \cite{Buckdahn_Quincampoix_Tessitore_2006}, \cite{GoreacMartinez2015}, \cite{LQiYongJiongminZhangXu_2012}). With this definition, equivalence of (global) approximate and approximate null-controllability can be established in various frameworks (constant coefficients \cite{Buckdahn_Quincampoix_Tessitore_2006}, \cite{GoreacMartinez2015}, continuous switching \cite{GoreacMartinez2015}). However, approximate reachability (or approximate terminal controllability) intrinsically requires the initial data to vary. For the frameworks studied in \cite{GoreacMartinez2015}, we show, by means of examples, that approximate terminal controllability (to every target) can hold without (global) approximate null-controllability. This seems to suggest that while approximate controllability (at least for the given frameworks) is captured by deterministic-related concepts (Kalman criterion, invariance, etc.), approximate terminal controllability is a purely stochastic concept. The study of approximate terminal controllability constitutes the core of Section \ref{SectionReach}. We begin with the LQ-formulation for controlled BSDE dynamics (the function $V$ described before). In order to proceed with our Riccati analysis, a further penalization (making appear the $\mathbb{L}^2$- norm of controls $u$) of $V$ is needed. We introduce the Riccati backward stochastic differential equation (hereafter called RBSDE) of interest in equation (\ref{RBSDE_ M}). While our LQ problem involves such equations with $0$ final data, the analysis is best conducted by passage to the limit in a non-degenerate framework. Links with the LQ problem and synthesis of the optimal control are given in the first main result Theorem \ref{ThMain}. The proofs rely on the structural representations (via a system of iterated deterministic Riccati equations (\ref{RiccatiDet})) in Section \ref{SectionStruct0} (see Proposition \ref{PropReductionDetSystem}). The second main result (Theorem \ref{TheoremSolvabilityDetSyst}) provides qualitative properties (existence, uniqueness, positiveness, lower and upper-bounds, monotony) of the equations intervening in (\ref{RiccatiDet}). The inherited properties for RBSDE are stated in Corollary \ref{CorStructuralPropSol}. Finally, all the proofs are gathered in Section \ref{SectionProofs}. 
\section{Preliminaries}
\label{SectionPreliminaries}
\subsection{General Notations}
We begin with recalling some elements on a particular class of pure-jump non-explosive processes living on the space $\Omega$. These processes are assumed to take their values (referred to as \textit{modes}) in a metric space denoted $E$ (and endowed with the associated Borel $\sigma$-field $\mathcal{B}(E)$) augmented by an isolated cemetery state $\Delta$ (i.e. $\bar{E}:=E\cup\set{\Delta}$. The interested reader may also want to consult \cite{Bremaud_1981} for a point process overview (or \cite{davis_93} for a strongly connected general framework concerning piecewise deterministic Markov processes). The \textit{mode process} denoted by $\Gamma$ will be given by 
\begin{itemize}
\label{lambdaQ}
\item[i.] a continuous, bounded transition intensity $\lambda:E\times\mathbb{R}_+\longrightarrow\mathbb{R}_+$
\item[ii.] a post-jump measure $Q:E\times\mathbb{R}_+\longrightarrow\mathcal{P}(E)$ such that, for all $\pr{\gamma,t}\in E\times\mathbb{R}_+$, one has $Q(\gamma,t,\set{\gamma})=0$ (i.e. no fictive jumps are allowed). Moreover, the application $E:\ni\gamma\mapsto Q\pr{\gamma,t,\cdot}$ is weakly continuous uniformly in $t\in\mathbb{R}_+$.
\end{itemize}
Here, $\mathcal{P}(E)$ stands for the set of all probability measures on $\pr{E,\mathcal{B}(E)}$.
We fix a finite number of jumps $J>1$. Given the initial mode $\gamma_0\in E$ (and fixing the associated probability measure $\mathbb{P}^{\gamma_0}$; since $\gamma_0$ will be considered to be fixed, by abuse of notation, we will drop the superscript and write $\mathbb{P}$), the first jump denoted by $T_1$ satisfies $\mathbb{P}\pr{T_1\geq t}=e^{-\int_0^t\lambda\pr{\gamma_0,r}dr}$. The mode process is then set to be $\Gamma_t^{\gamma_0}:=\gamma_0$ on $0\leq t<T_1$. The post-jump position $\gamma_1$ is chosen using $Q$ as conditional distribution i.e. $\mathbb{P}\pr{\gamma_1\in A \mid T_1}=Q\pr{\gamma_0,T_1,A}$. If $J\geq 2$, the inter-jump time $S_2:=T_2-T_1$ is distributed as $\mathbb{P}\pr{T_2-T_1\geq t\mid \gamma_1,T_1}=e^{-\int_0^t \lambda\pr{\gamma_1,T_1+r}dr}$ and the post)jump position $\gamma_2$ is given by $\mathbb{P}\pr{\gamma_2\in A\mid \gamma_1,T_2}=Q\pr{\gamma_1,T_2,A}$. And so on. At $J$-th jump, the process is stopped by setting $\gamma_J:=\Delta$ and extending $\lambda\pr{\Delta,t}=0$.
\begin{remark}
As we will see later on, the process of interest $\pr{\Gamma,X}$ will no longer be Markovian such that the Markov construction of $\Gamma$ plays little in the arguments. Although this construction will be sufficient for the examples we have in mind, extensions to intensities (for say $n$-th jump of type) $\lambda^j\pr{\gamma_0, 0, \gamma_1,t_1,\ldots,\gamma_{j-1},r}$ and similar post-jump measures are treated with similar arguments.
\end{remark}
To such processes, we associate the filtration $\mathbb{F}$ given by $\pr{\mathcal{F}_{\pp{0,t}}:=\sigma\set{\Gamma_s^{\gamma_0} : s\in \pp{0,t}}}_{t\geq 0}$. As usual, we let $\mathcal{P}$ stand for the predictable $\sigma$-field while $\mathcal{P}rog$ will stand for the progressive $\sigma$-field. The associated random measure on $\Omega\times \mathbb{R}_+ \times E$ is given by $q\pr{\omega,A}:=\sum_{j\geq 1}\mathbbm{1}_{\pr{T_j(\omega),\Gamma_{T_j(\omega)}^{\gamma_0}}\in A}$, for all $A\in \mathcal{B}\pr{\mathbb{R}_+}\times\mathcal{B}(E)$ and all $\omega\in\Omega$. The compensator of $q$ will be denoted by $\hat{q}(dt,d\theta)$. Finally, we will use the compensated measure $\tilde{q}:=q-\hat{q}$.
We will make use of the following notations 
\begin{itemize}
\item Throughout the paper, unless stated otherwise, $T>0$ will be a \textit{fixed finite time horizon}.
\item For a (generic) Euclidean space $\mathcal{E}$, we denote by $\scal{\cdot,\cdot}$ its \textit{scalar product} and by $\norm{\cdot}$ the induced \textit{norm}. Thus, if $n,m$ are positive integers,
\item $\mathbb{R}^n$ will stand for the \textit{standard n-dimensional Euclidean space} endowed with the usual Euclidean structure;
\item $\mathbb{R}^{n\times m}$ stands for the \textit{space of $n\times m$-type matrices} endowed with the scalar product $\scal{\alpha,\beta}:=Tr\pp{\beta^*\alpha}$, where $\cdot^*$ stands for \textit{transposition} and $Tr$ for the \textit{trace operator};
\item For square matrices (i.e. $m=n$), we denote by $I_n$ the (diagonal) identity matrix and by $0_n$ the matrix whose elements are $0$. Moreover, we let $\mathcal{S}^n$ be the space of symmetric square matrices and $\mathcal{S}^n_+$ the space of symmetric positive semi-definite square matrices.
\item For a matrix $A\in \mathbb{R}^{n\times m}$, we will denote by $Rank[A]$ its \textit{rank} (i.e. the dimension of the space spanned by its columns), by $ker\pr{A}$ its \textit{kernel} i.e. $ker\pr{A}:=\set{x\in\mathbb{R}^m:\ Ax=0}$;
\item if $\pp{t_1,t_2}\subset\mathbb{R}_+$, then $\mathbb{D}\pr{\pp{t_1,t_2};\mathcal{E}}$ stands for the space of càdlàg functions $\phi$ endowed with the supremum norm $\norm{\phi}_0:=\sup_{t\in \pp{t_1,t_2}}\norm{\phi(t)}$;
\item if $\pp{t_1,t_2}\subset\mathbb{R}_+$, then $\mathbb{C}\pr{\pp{t_1,t_2};\mathcal{E}}$ stands for the space of continuous functions $\phi$ endowed with the supremum norm.
\end{itemize}
Considering a finite (fixed) time horizon $T>0$, some $\pp{t_1,t_2}\subset \pp{0,T}$ and for $p\in\left[1,\infty\right]$,
\begin{itemize}
\item the space $\mathbb{L}^p\pr{\Omega,\mathcal{F}_{\pp{0,t_1}},\mathbb{P};\mathcal{E}}$ stands for the usual $p$-integrable space of equivalence classes admitting an $\mathcal{F}_{\pp{0,t_1}}$-measurable representative endowed with the usual norms $\norm{\phi}_p^p:=\mathbb{E}\pp{\norm{\phi}^p}$ (resp. $\norm{\phi}_{\infty}:=\textnormal{esssup}_{\omega\in \Omega}\norm{\phi(\omega)}$);
\item the space $\mathbb{L}^p\pr{\Omega\times \pp{t_1,t_2},\mathcal{P},\mathbb{P}\times\mathcal{L}eb;\mathcal{E}}$ stands for the usual $p$-integrable space endowed with the usual norms $\norm{\Phi}_p^p:=\mathbb{E}\pp{\int_{t_1}^{t_2} \norm{\Phi_r}^p dr}$ (resp. $\norm{\Phi}_{\infty}:=\textnormal{esssup}_{\pr{\omega,s}\in \Omega\times \pp{t_1,t_2} }\norm{\Phi_s(\omega)}$);
\item the space of admissible $d$-dimensional control processes on $\pp{0,T}$ will be \[\mathcal{U}:=\mathbb{L}^2\pr{\Omega\times \pp{0,T},\mathcal{P},\mathbb{P}\times\mathcal{L}eb;\mathbb{R}^d};\]
\item the space $\mathbb{L}^p\pr{\Omega;\mathbb{D}\pr{\pp{t_1,t_2};\mathcal{E}}}$ is to be understood as (classes of) progressively measurable processes having càdlàg paths and endowed with finite adequate norm i.e. \\$\norm{\Phi}_{\mathbb{L}^p\pr{\Omega;\mathbb{D}\pr{\pp{t_1,t_2};\mathcal{E}}}}^p:=\mathbb{E}\pp{\norm{\Phi_\cdot}_0^p}$ (resp. $\norm{\Phi}_{\mathbb{L}^{\infty}\pr{\Omega;\mathbb{D}\pr{\pp{t_1,t_2};\mathcal{E}}}}:=\textnormal{esssup}_{\omega\in\Omega}\pp{\norm{\Phi_\cdot(\omega)}_0}$.
\item the space $\mathbb{L}^p\pr{\Omega;\mathbb{C}\pr{\pp{t_1,t_2};\mathcal{E}}}$ is to be understood as a subspace of $\Phi\in\mathbb{L}^p\pr{\Omega\times [0,t],\mathcal{P},\mathbb{P}\times\mathcal{L}eb;\mathcal{E}}$ having continuous paths and endowed with finite adequate norm i.e. \\$\norm{\Phi}_{\mathbb{L}^p\pr{\Omega;\mathbb{C}\pr{\pp{t_1,t_2};\mathcal{E}}}}^p:=\mathbb{E}\pp{\norm{\Phi_\cdot}_0^p}$ (resp. $\norm{\Phi}_{\mathbb{L}^{\infty}\pr{\Omega;\mathbb{C}\pr{\pp{t_1,t_2};\mathcal{E}}}}:=\textnormal{esssup}_{\omega\in\Omega}\pp{\norm{\Phi_\cdot(\omega)}_0}$;
\item the space $\mathbb{L}^p\pr{\Omega\times \pp{t_1,t_2}\times E,q;\mathcal{E}}$ is the space of all (classes of) processes $\Phi:\Omega\times \pp{t_1,t_2}\times E\longrightarrow \mathcal{E}$ that are $\mathcal{P}\otimes\mathcal{B}(E)$-measurable (i.e. \textit{predictable}) endowed with the norm $\norm{\Phi}_{\mathbb{L}^p\pr{\Omega\times \pp{t_1,t_2}\times E,q;\mathcal{E}}}^p:=\mathbb{E}\pp{\int_{\pp{t_1,t_2}\times E}\norm{\Phi_r\pr{\theta}}^pq(dr,d\theta)}<\infty.$
\end{itemize}
\subsection{Controlled Systems. Controllabilities}
We consider the forward controlled piecewise linear switched system
\begin{equation}\label{eq0}
\left\lbrace
\begin{split}
dX_{r}^{x,u}&=\left(A_{r}X_{r}^{x,u}+B_{r}u_{r}\right)dr+\int_{E}C_{r}(\theta)X_{r}^{x,u}\tilde{q}(dr,d\theta),\textnormal{ for all }r\geq0,\\
X_{0}^{x,u}&=x\in\mathbb{R}^n.
\end{split}
\right.
\end{equation}
The coefficients are assumed to satisfy some standard regularity (measurability and integrability) properties.
\begin{assumption}
\label{AssCoeff}Throughout the paper, unless stated otherwise, we assume that
\begin{equation} 
\begin{split}
&A\in \mathbb{L}^\infty\pr{\Omega\times \pp{0,T},\mathcal{P},\mathbb{P}\times\mathcal{L}eb;\mathbb{R}^{n\times n}};\\ 
&B\in \mathbb{L}^\infty\pr{\Omega\times \pp{0,T},\mathcal{P},\mathbb{P}\times\mathcal{L}eb;\mathbb{R}^{n\times d}};\\
&C\in \mathbb{L}^2\pr{\Omega\times \pp{t_1,t_2}\times E,q;\mathbb{R}^{n\times n}}\textnormal{ and }\norm{C}_\infty:=\textnormal{esssup}_{\omega\in\Omega}\sup_{r\in\pp{0,T},\theta\in E}\norm{C_r\pr{\omega,\theta}}<\infty.
\end{split}
\end{equation}
We denote by $\norm{A}_\infty:=\norm{A}_{\mathbb{L}^\infty\pr{\Omega\times \pp{0,T},\mathcal{P},\mathbb{P}\times\mathcal{L}eb;\mathbb{R}^{n\times n}}},\ \norm{B}_\infty:=\norm{B}_{\mathbb{L}^\infty\pr{\Omega\times \pp{0,T},\mathcal{P},\mathbb{P}\times\mathcal{L}eb;\mathbb{R}^{n\times d}}}.$
\end{assumption}
\begin{remark}
To cope with the finite number of jumps, we assume $A_t=0_{n},\ B_t=0_{n\times d},\textnormal{ and }C_t(\cdot)=0_n,\textnormal{ on }\set{t>T_J}.$
\end{remark}
We recall that a process $u$ is said to be \textit{admissible} (up to the fixed time horizon $T>0$) if \[u\in \mathbb{L}^2\pr{\Omega\times \pp{0,T},\mathcal{P},\mathbb{P}\times\mathcal{L}eb;\mathbb{R}^d}.\] To simplify notations in the infimum expressions, this space (of admissible controls) is also denoted by $\mathcal{U}$. A solution to (\ref{eq0}) is a process $X^{x,u}\in\mathbb{L}^2\pr{\Omega;\mathbb{D}\pr{\pp{0,T};\mathbb{R}^n}}$ such that, $\mathbb{P}$-a.s., \[X_t^{x,u}=\int_0^t\pr{A_r X_r^{x,u}+B_r u_r}dr+\int_E\int_0^tC_r(\theta)X_r^{x,u}\tilde{q}(dr,d\theta),\textnormal{ for all }t\in\pp{0,T}.\]Standard results on SDEs (e.g. in \cite{Ikeda_Watanabe_1981}) yield, for every initial datum $x\in\mathbb{R}^n$ and every admissible $u$, the existence and uniqueness of the solution $X_\cdot^{x,u}$. Moreover, for some generic constant $c$ (depending only on $T,\ \norm{A}_\infty,\ \norm{B}_\infty,\ \norm{C}_\infty$ but neither of $u$ nor of $x$) such that \[\mathbb{E}\pp{\sup_{s\in\pp{0,t}}\norm{X_s^{x,u}}^2}\leq c\pp{\norm{x}^2+\mathbb{E}\pp{\int_0^t\norm{u_s}^2ds}}.\]
We recall (or introduce) the following controllability notions constituting the core of our paper. We only focus on the weakest notions of \textit{approximate} controllability. Absence of (global) exact or exact terminal controllability for control-free noise coefficients (as it is our case) go back to \cite{Merton_76} respectively \cite{Peng_94}.
\begin{definition}
\label{DefCtrl}
\begin{itemize}
\item[i. ]The system (\ref{eq0}) is \textbf{approximately controllable in time $T$ from $x\in\mathbb{R}^n$ to $\xi\in \mathbb{L}^2\pr{\Omega,\mathcal{F}_{\pp{0,T}},\mathbb{P};\mathbb{R}^n}$} if, for every $\varepsilon>0$, there exists an admissible control $u^\varepsilon$ such that \[\norm{X_T^{x,u^\varepsilon}-\xi}_{\mathbb{L}^2\pr{\Omega,\mathcal{F}_{\pp{0,T}},\mathbb{P};\mathbb{R}^n}}\leq \varepsilon.\] If the condition holds true from every initial datum to every final random variable, the is said to be \textbf{approximately controllable in time $T$}.
\item[ii. ]The system (\ref{eq0}) is \textbf{approximately null-controllable in time $T$ from $x\in\mathbb{R}^n$} if it is approximately controllable in time $T$ from $x$ to $\xi=0$. It is \textbf{approximately null-controllable in time $T$} if the above holds for every $x\in \mathbb{R}^n$.
\item[iii. ]The system (\ref{eq0}) is \textbf{approximately terminal controllable to $\xi\in \mathbb{L}^2\pr{\Omega,\mathcal{F}_{\pp{0,T}},\mathbb{P};\mathbb{R}^n}$} if, for every $\varepsilon>0$, there exists an initial datum $x^\varepsilon\in\mathbb{R}^n$ and an admissible control $u^\varepsilon$ such that, $\norm{X_T^{x^\varepsilon,u^\varepsilon}-\xi}_{\mathbb{L}^2\pr{\Omega,\mathcal{F}_{\pp{0,T}},\mathbb{P};\mathbb{R}^n}}\leq \varepsilon$. Finally, if this condition holds for every $\xi\in \mathbb{L}^2\pr{\Omega,\mathcal{F}_{\pp{0,T}},\mathbb{P};\mathbb{R}^n}$, the system (\ref{eq0}) is said to be \textbf{approximately terminal controllable}.
\item[iv. ]Alternatively, we will say that $\xi\in \mathbb{L}^2\pr{\Omega,\mathcal{F}_{\pp{0,T}},\mathbb{P};\mathbb{R}^n}$ is \textbf{approximately reachable} whenever (\ref{eq0}) is approximately terminal controllable to $\xi$.
\end{itemize}
\end{definition}
\section{Basic Results and Examples}
\label{SectionBasicExp}
The aim of this section is twofold. First, we state some abstract results characterizing the controllability concepts. As consequence, approximate terminal controllability is proven to provide the missing link between approximate null and approximate controllability (required for all initial data). Second, we give some examples illustrating the stochastic essence of the concept.
\subsection{Range/Kernel Approach and Abstract Results}
It is by now a classical approach to consider the following linear controllability operators 
\begin{equation}
\left\lbrace
\begin{split}
L_{1}:\mathcal{U}\longrightarrow \mathbb{L}^2\pr{\Omega,\mathcal{F}_{\pp{0,T}},\mathbb{P};\mathbb{R}^n},&\textnormal{  given by }L_1(u):=X_T^{0,u} \textnormal{ the solution of }(\ref{eq0})\\ &\textnormal{ starting at $0$ and using the control $u$ and}\\
L_2:\mathbb{R}^n\longrightarrow \mathbb{L}^2\pr{\Omega,\mathcal{F}_{\pp{0,T}},\mathbb{P};\mathbb{R}^n},&\textnormal{  given by }L_2(x):=X_T^{x,0} \textnormal{ the solution of }(\ref{eq0}) \\&\textnormal{ starting from $x$, with $0$ control.}
\end{split}
\right.
\end{equation}
\begin{remark}
Using these operators and $L:=\left[L_1 \ L_2\right]$ (defined on $\mathcal{U}\times\mathbb{R}^n$ and given by $L(u,x):=L_1u+L_2x$ if $\pr{u,x}\in \mathcal{U}\times\mathbb{R}^n$), the notions of controllability in Definition \ref{DefCtrl} become respectively
\begin{itemize}
\item[i. ]$\xi-L_2x\in cl\pr{Range\pr{L_1}}$ where $cl$ is the Kuratowski closure operator with respect to the topology of $\mathbb{L}^2\pr{\Omega,\mathcal{F}_{\pp{0,T}},\mathbb{P};\mathbb{R}^n}$;
\item[ii. ]$L_2x\in cl\pr{Range\pr{L_1}}$;
\item[iii. ]$\xi\in cl\pr{Range\pr{L}}$ (again closure is intended with respect to $\mathbb{L}^2\pr{\Omega,\mathcal{F}_{\pp{0,T}},\mathbb{P};\mathbb{R}^n}$).
\end{itemize}
\end{remark}
It is obvious that all these properties can equally be given with respect to the dual operator of $L$. This dual turns out to be strongly connected to the solution of the following backward (dual) stochastic differential equation.
\begin{equation}\label{DualBSDE0}
\begin{cases}
d\mathcal{X}_t^{T,\zeta}=-A_t^*\mathcal{X}_t^{T,\zeta}dt-\int_EC_t^*(\theta)\mathcal{Z}_t^{T,\zeta}(\theta)\hat{q}(dt,d\theta)+\int_E\mathcal{Z}_t^{T,\zeta}(\theta)\tilde{q}(dt,d\theta), \textnormal{for all }0\leq t\leq T,\\
\mathcal{X}_T^{T,\zeta}=\zeta\in\mathbb{L}^2\pr{\Omega,\mathcal{F}_{\pp{0,T}},\mathbb{P};\mathbb{R}^n} .
\end{cases}
\end{equation}
We recall (following \cite{Confortola_Fuhrman_2014}), that a solution of such an equation consists of a couple \[\pr{\mathcal{X}_\cdot^{T,\zeta},\mathcal{Z}_\cdot^{T,\zeta}(\theta)}\in \mathbb{L}^2\pr{\Omega;\mathbb{D}\pr{\pp{0,T};\mathbb{R}^n}}\times\mathbb{L}^2\pr{\Omega\times \pp{0,T}\times E,q;\mathbb{R}^n}\] satisfying, $\mathbb{P}\times \mathcal{L}eb$-a.s.
\begin{equation*}
\zeta=\mathcal{X}_t^{T,\zeta}-\int_t^TA_r^*\mathcal{X}_r^{T,\zeta}dr-\int_t^T\int_EC_r^*(\theta)\mathcal{Z}_r^{T,\zeta}(\theta)\hat{q}(dr,d\theta)+\int_t^T\int_E\mathcal{Z}_r^{T,\zeta}(\theta)\tilde{q}(dr,d\theta).
\end{equation*}Existence and uniqueness follow from \cite{Confortola_Fuhrman_2014} under the assumption \ref{AssCoeff}.\\
Having given these details on the backward stochastic system (\ref{DualBSDE0}), the reader is invited to note the following.  The dual operator of $L$ is given by \[L^*:\mathbb{L}^2\pr{\Omega,\mathcal{F}_{\pp{0,T}},\mathbb{P};\mathbb{R}^n}\longrightarrow\mathcal{U}\times\mathbb{R}^n, \ L^*=\left(\begin{array}{c}L_1^* \\ L_2^*\end{array}\right),\textnormal{ where } L_1(\zeta)=\left(B_t^*\mathcal{X}_t^{T,\zeta}\right)_{0\leq t \leq T}\textnormal{ and }L_2(\zeta)=\mathcal{X}_0^{T,\zeta},\]
for all $\zeta\in\mathbb{L}^2\pr{\Omega,\mathcal{F}_{\pp{0,T}},\mathbb{P};\mathbb{R}^n}$. This assertion is quite classical and its proof relies on a mere application of Itô's formula on $[0,T]$ to the Euclidian product $\left\langle X_{\cdot}^{x,u},\mathcal{X}_{\cdot}^{T,\zeta}\right\rangle$ (the interested reader may take a look at \cite[proof of Theorem 1]{GoreacMartinez2015}. 
As consequence, we get the following result.
\begin{proposition}\label{PropATC_abstract}
\begin{itemize}
\item[1] \cite[Theorem 1]{GoreacMartinez2015}. The system $(\ref{eq0})$ is \\
- approximately controllable in time $T$ if and only if the only solution to $(\ref{DualBSDE0})$ satisfying $
B_t^*\mathcal{X}_t^{T,\zeta}=0, \mathbb{P}\times\mathcal{L}eb\textnormal{-a.s. on }\Omega\times[0,T]
$
is the trivial (zero) solution;\\
- approximately null-controllable in time $T$ if and only if any solution to $(\ref{DualBSDE0})$ satisfying $
B_t^*\mathcal{X}_t^{T,\zeta}=0, \mathbb{P}\times\mathcal{L}eb\textnormal{-a.s. on }\Omega\times[0,T]
$ satisfies, $\mathbb{P}-a.s.$, $\mathcal{X}_0^{T,\zeta}=0$.\\
\item[2. ]The system $(\ref{eq0})$ is approximately terminal controllable to the target $\xi\in\mathbb{L}^2\left(\Omega,\mathcal{F}_{\pp{0,T}},\mathbb{P};\mathbb{R}^n\right)$ if and only if any solution to $(\ref{DualBSDE0})$ satisfying $
B_t^*\mathcal{X}_t^{T,\zeta}=0, \mathbb{P}\times\mathcal{L}eb\textnormal{-a.s. on }\Omega\times[0,T]$ and $\mathcal{X}_0^{T,\zeta}=0$
also satisfies $\mathbb{E}\left[\left\langle\zeta,\xi\right\rangle\right]=0.$
\item[3. ] In particular, the system $(\ref{eq0})$ is approximately terminal controllable to every target if and only if any solution to $(\ref{DualBSDE0})$ satisfying $B_t^*\mathcal{X}_t^{T,\zeta}=0, \mathbb{P}\times\mathcal{L}eb\textnormal{-a.s. on }\Omega\times[0,T]$ and $\mathcal{X}_0^{T,\zeta}=0$ also satisfies $\zeta=0$, $\mathbb{P}$-a.s. on $\Omega$.
\end{itemize}
\end{proposition}
We only need to prove assertions 2 and 3. The proof is quite standard and postponed to Section \ref{SectionProofs}.\\
It has already been shown that approximate controllability is, in general, strictly stronger than approximate null-controllability (cf. \cite[Example 9]{GoreacGrosuRotenstein_2016}). As a simple consequence of Proposition \ref{PropATC_abstract}, it turns out that the missing link between (global) approximate controllability and terminal null-controllability is the notion of approximate terminal controllability.
\begin{cor}
\label{CorAbstractAppTerminalCtrl}
The system  $(\ref{eq0})$ is approximately controllable (starting from every $x\in\mathbb{R}^n$ to every target $\xi\in\mathbb{L}^2\left(\Omega,\mathcal{F}_{\pp{0,T}},\mathbb{P};\mathbb{R}^n\right)$) if and only if the following two assertions hold simultaneously:\\
(i) the system $(\ref{eq0})$ is approximately terminal controllable to every target $\xi\in\mathbb{L}^2\left(\Omega,\mathcal{F}_{T},\mathbb{P};\mathbb{R}^n\right)$ and\\
(ii) it is also approximately null-controllable (starting from every $x\in\mathbb{R}^n$).
\end{cor}
\begin{remark}
\item[i.] In fact, one can prove a slightly stronger condition concerning the sufficiency. If (ii) holds true then the system  $(\ref{eq0})$ is  approximately controllable to a given target $\xi\in\mathbb{L}^2\left(\Omega,\mathcal{F}_{T},\mathbb{P};\mathbb{R}^n\right)$ (starting from every $x$) if and only if the system $(\ref{eq0})$ is approximately terminal controllable to the target $\xi$ starting from some $x_0\in\mathbb{R}^n$.
\item[ii.] As we have already mentioned, \cite[Example 9]{GoreacGrosuRotenstein_2016} provides an example of approximately null-controllable system that is not approximately controllable (or, in other words, not approximately terminal controllable to some directions specified in this example).
\end{remark}
\subsection{Approximate Terminal Controllability, A Purely Sochastic Concept}
Let us begin with a simple remark in the deterministic framework (i.e. whenever no switch occurs) with constant coefficients. In this case, the solution of (\ref{eq0}) is explicitly given by \[X_t^{x,u}=e^{At}x+\int_0^te^{A(t-s)}Bu_sds.\] The dual component satisfies $\mathcal{X}_t^{T,\zeta}=e^{A^*(T-t)}\zeta,\textnormal{ where }\zeta\in\mathbb{R}^n.$ The approximate terminal controllability in this case becomes $\ker\pp{e^{A^*T}}\cap\pr{\cap_{0\leq k\leq n-1}\ker \pp{B^*\pr{A^*}^k}}\subset\set{x\in\mathbb{R}^n:\scal{x,\xi}=0}$. Since $e^{A^*T}$ is invertible (for all $T>0$), $\ker\pp{e^{A^*T}}$ reduces to $0$ and it follows that any such system is approximately terminal controllable. (In fact, it suffices to take $x:=e^{-AT}\xi,\ u=0$ to drive the solution to $\xi\in\mathbb{R}^n$.) This explains why approximate terminal controllability is a \textit{purely stochastic} concept.\\
Let us now consider the simplest (Markovian) case, in which the coefficients are given as deterministic functions of the mode (i.e. $A_t=A\pr{\Gamma_t},\ B_t=B\pr{\Gamma_t},\ C_t\pr{\cdot}=C\pr{\Gamma_t,\cdot}$).\footnote{The reader is invited to note that we make a slight abuse of notation by considering a measurable function $A: E \longrightarrow \mathbb{R}$ and then setting $A(\omega,t):=A\pr{\Gamma_t(\omega)}$. This kind of abuse of notation will be employed several times in the sequel (particularly for examples).} 
\subsubsection{Continuous Switching}
 If $C=0$, the system has no jump affecting the $X$ component. This corresponds to randomly switching the linear (deterministic) systems. Since this case is a (very) slight generalization of the deterministic setting, one is entitled to ask if (at least) deterministic targets can be approached with every such system. The negative answer is provided by the following.
 \begin{example}
 \label{Exp1}
We consider a one-dimensional state and control space (i.e. $n=d=1$), the mode space $E=\set{0,1}$, the transition intensity $\lambda=1$ and the transition probability $Q\pr{\gamma,\set{1-\gamma}}=1$, for $\gamma\in E$. The coefficients are given by $A(\gamma)=\gamma,\ B=0$. The explicit solution (note that $u$ has no impact) of (\ref{eq0}) using the initial mode $\gamma_0=0$ is \[X_t^x=x\pr{\mathbbm{1}_{T_1\geq t}+e^{\min\pr{t,T_2}-T_1}\mathbbm{1}_{T_3\geq t}+e^{\min\pr{t,T_4}-T_3+T_2-T_1}\mathbbm{1}_{T_5\geq t}+\ldots}\textnormal{, for }t\geq 0.\]Since $\mathbb{P}\pr{T_1>T}>0$ it follows that, in order to envisage the target $\xi >0$, one starts with $x=\xi$. But, in this case, $X_t^\xi\geq\xi$ and, on $T_1\leq \frac{T}{2}\leq T\leq T_2$, one has $X_t^x\geq \xi e^\frac{T}{2}$. Since $\mathbb{P}\pr{T_1\leq \frac{T}{2},\ T\leq T_2}>0$, we get that  $X_T^\xi\neq \xi$.
 \end{example}
Second, for continuous switching systems (cf. \cite[Section 4.2]{GoreacMartinez2015}), approximate null-controllability (starting from every $x\in\mathbb{R}^n$) and approximate controllability are equivalent. The criterion is a Kalman-type condition holding true for every deterministic component of the dynamics i.e. the full rank condition \begin{equation}
\label{Kalman}
\pp{B(\gamma), A(\gamma)B(\gamma),\ldots,A^{n-1}(\gamma)B(\gamma)}=n.
\end{equation} should hold true for every $\gamma\in E$ (accessible from $\gamma_0$). However, when one can have approximate terminal controllability (to all targets) without having these deterministic conditions satisfied.  
 \begin{example}
 \label{Exp2}
We consider the state space to be $\mathbb{R}^2$, the mode space $E=\set{0,1}$, the transition intensity $\lambda=1$ and the transition probability $Q\pr{\gamma,\set{1-\gamma}}=1$, for all $\gamma\in E$ and $A(\gamma)=B(\gamma)=\gamma I_2$.
It is clear that for $\gamma=0$, the couple $\pp{A(0),B(0)}$ fails to satisfy Kalman's criterion (see (\ref{Kalman})). \\
The solution to the dual system satisfies \[d\mathcal{X}_t^{T,\zeta}=\pp{-A_t^*\mathcal{X}_t^{T,\zeta}-\mathcal{Z}_t^{T,\zeta}(1-\gamma_t)}dt+\int_E\mathcal{Z}_t^{T,\zeta}(\theta)q(dt,d\theta).\]
The reader will note that $\ker\ B^*(\gamma)=(1-\gamma)\mathbb{R}^2$. We consider the initial mode $\gamma_0=0$. Since $\mathcal{X}_{T_1}^{T,\zeta}=0,\ \mathbb{P}-a.s.$ (recall that $\mathcal{X}$ has càdlàg trajectories and $\ker B^*(1)=\set{0}$), it follows that $\mathcal{Z}_t^{T,\zeta}(0)=-\mathcal{X}_t^{T,\zeta}$ (at least prior to $T_1$). By plugging this feedback form into the equation, as soon as $\mathcal{X}_0^{T,\zeta}=0$, one gets $\mathcal{X}_t^{T,\zeta}=0$ on $\left[ 0,T_1 \right]$. On $\left[T_1,T_2\right)$, since $\ker B^*(1)=\set{0}$, one gets $\mathcal{X}_t^{T,\zeta}=0,\ \mathbb{P}-a.s.$. Thus, $\mathcal{Z}_t^{T,\zeta}=0$ on the same set. It follows that, $\mathcal{X}_{T_2}^{T,\zeta}=0,\ \mathbb{P}-a.s.$ and the arguments can be repeated up to $T$ in order to conclude that $\mathcal{X}^{T,\zeta}=0$. Owing to Corollary \ref{CorAbstractAppTerminalCtrl}, one gets the approximate terminal controllability to any given target. 
 \end{example}
 In conclusion, even for these simple systems, the notion of approximate terminal controllability is not trivial (as it is the case in a purely deterministic framework) and it is strictly weaker than approximate null-controllability.  
 \subsubsection{The Constant Coefficients Case}\label{CosntCoeffCase}
For constant coefficients $A$, $B$, $C$, approximate and approximate null-controllability (for all initial data) are also equivalent (cf. \cite[Section 4.2]{GoreacMartinez2015}). The criterion is again deterministic (reducing to invariant subspaces of the kernel $\ker\pr{B^*}$). However, even in this case, approximate terminal controllability is strictly weaker than approximate null-controllability (for all initial datum).
\begin{example} \label{Exp3}
We consider a two-dimensional state space ($n=2$), a one-dimensional control space ($d=1$) and constant coefficients given by \[A:=\begin{pmatrix}0 & 1 \\ 0 & 0 \end{pmatrix},  C=\begin{pmatrix}-1 & \frac{1}{2} \\ 0 & -1 \end{pmatrix},  B:=\begin{pmatrix}0 \\ 1 \end{pmatrix}.\] The mode space is $E:=\set{0,1}$, the intensity $\lambda=1$ and the system switches between the two states in $E$ (i.e. $Q(\gamma)=\delta_{\set{1-\gamma}}$, where $\delta_\cdot$ stands for the usual Dirac mass). 
\begin{itemize}
\item It is easy to see that the associated system is not approximately null-controllable. To this purpose, we set $\zeta:=\begin{pmatrix}
(-1)^{q(\pp{0,T}\times E)}\\0
\end{pmatrix}$. One gets a solution of (\ref{DualBSDE0}) given by the couple $\mathcal{X}_t^{T,\zeta}:=\begin{pmatrix}
(-1)^{q(\pp{0,t}\times E)}\\0
\end{pmatrix}$ and $\mathcal{Z}_t^{T,\zeta}(\theta):=2\begin{pmatrix}
(-1)^{(1+q(\left[0,t\right)\times E))}\\0
\end{pmatrix}$. By noting that $\ker B^*=span\left\lbrace\begin{pmatrix}
1\\0
\end{pmatrix}\right\rbrace$, one has found a solution of the dual system (\ref{DualBSDE0}) remaining in the $\ker B^*$ but that does not trivially reduce to ${0}$. The conclusion follows from Proposition \ref{PropATC_abstract}, assertion 1.
\item On the other hand, if one further imposes $\mathcal{X}_0^{T,\zeta}=0$ and asks that the solution of the dual system (\ref{DualBSDE0}) remains in $\ker B^*=span\left\lbrace\begin{pmatrix}
1\\0
\end{pmatrix}\right\rbrace$, it follows that $\mathcal{X}_t^{T,\zeta}=0$ on $\set{t<T_1}$. As consequence, $\mathcal{Z}_t^{T,\zeta}=\begin{pmatrix}
0\\ z_t^{T,\zeta}
\end{pmatrix}$ on $\set{t\leq T_1}$. Recalling that $\mathcal{X}_{T_1}^{T,\zeta}=\mathcal{Z}_{T_1}^{T,\zeta}\in \ker B^*=span\left\lbrace\begin{pmatrix}
1\\0
\end{pmatrix}\right\rbrace$, one deduces that $\mathcal{X}_{T_1}^{T,\zeta}=0$ and the argument can be repeated up to time $T$. Owing to Proposition \ref{PropATC_abstract} (assertion 3.), this implies that the considered system is approximately terminal controllable.
\end{itemize}
\end{example}
\section{Approximately Reachable Data}
\label{SectionReach}
\subsection{A Control Problem With BSDE Dynamics Formulation}
As we have seen, characterizing random variables that are reachable allows one to close the gap between (approximate) controllability to $0$ and approximate controllability to such targets. 
On the other hand, such terminal data $\xi\in \mathbb{L}^2\left(\Omega,\mathcal{F}_{\pp{0,T}},\mathbb{P};\mathbb{R}^n\right)$ can be exactly represented by using backward stochastic differential equations. To this purpose, let us fix the terminal time $t$ such that $0\leq t\leq T$ and consider
a further modification in a form of a standard backward stochastic differential equation by dropping the initial datum $x$ and adding a $\mathbf{Z}$ term to get
\begin{equation}
\label{BSDE1}
\left\lbrace
\begin{split}
d\mathbf{X}_{r}^{u,t,\xi}&=\left(A_{r}\mathbf{X}_{r}^{u,t,\xi}+B_{r}u_{r}\right)dr+\int_{E}C_{r}(\theta)\mathbf{X}_{r}^{u,t,\xi}\tilde{q}(dr,d\theta)+\int_{E}\mathbf{Z}_{r}^{u,t,\xi}(\theta)\tilde{q}(dr,d\theta),\textnormal{ for all }r\geq0,\\
\mathbf{X}_{t}^{u,t,\xi}&=\xi\in\mathbb{L}^2\left(\Omega,\mathcal{F}_{\pp{0,t}},\mathbb{P};\mathbb{R}^n\right).
\end{split}
\right.
\end{equation}As before, by solution to (\ref{BSDE1}) we understand a couple of processes $\pr{\mathbf{X}_{\cdot}^{u,t,\xi},\mathbf{Z}_{\cdot}^{u,t,\xi}}\in\mathbb{L}^2\pr{\Omega;\mathbb{D}\pr{\pp{0,t};\mathbb{R}^n}}\times\mathbb{L}^2\pr{\Omega\times \pp{0,t}\times E,q;\mathbb{R}^n}$ such that \[\xi=\mathbf{X}_{r}^{u,t,\xi}+\int_r^t\left(A_{s}\mathbf{X}_{s}^{u,t,\xi}+B_{s}u_{s}\right)ds+\int_r^t\int_{E}C_{s}(\theta)\mathbf{X}_{s}^{u,t,\xi}\tilde{q}(ds,d\theta)+\int_r^t\int_{E}\mathbf{Z}_{s}^{u,T,\xi}(\theta)\tilde{q}(ds,d\theta).\] Existence and uniqueness follow from \cite{CFJ_2014} under the assumption \ref{AssCoeff}.\\
At this point, it is easy to see that (approximately terminally) controlling the system (\ref{eq0}) to $\xi$ amounts to asking the system (\ref{BSDE1}) to be solvable (with final data $\xi$) and "almost $0$" component $\mathbf{Z}^{u,t,\xi}$. This justify considering the value function
\begin{equation}\label{value}
V(t,\xi):=\inf_{u \in\mathcal{U}}J(t,\xi,u),
\textnormal{ where } J(t,\xi,u):=\mathbb{E}^{\gamma_0}\left[ \int_{0}^{t}\Vert \mathbf{Z}_{r}^{u,t,\xi}(\theta)\Vert ^2\hat{q}(dr,d\theta) \right].
\end{equation}
We get the following result.
\begin{proposition}
The random datum $\xi\in\mathbb{L}^{2}\left(\Omega,\mathcal{F}_{\pp{0,t}},\mathbb{P};\mathbb{R}^{n}\right)$ is approximately reachable in time $0\leq t\leq T$ with trajectories governed by system (\ref{BSDE1})\footnote{or, equivalently, the system (\ref{BSDE1}) is approximately terminal controllable in time $t$ to $\xi$} if and only if $V(t,\xi)=0$.
\end{proposition}
From now on, we will fix $t=T$. An alternate formulation giving the same value function is based on a standard penalization by artificially introducing a non-degenerate component with respect to the control process. The need for non-degeneracy is quite standard in the synthesis of optimal control using Riccati equations (approach we have chosen to adopt hereafter). To this purpose, we consider, for every $N\geq 1$ the value function(s)
\begin{equation}
\begin{split}
V^N(T,\xi):=&\inf_{u \in \mathcal{U}}J^N(T,u,\xi), \textnormal {where }\xi\in\mathbb{L}^2\left(\Omega,\mathcal{F}_{\pp{0,T}},\mathbb{P};\mathbb{R}^n\right)\textnormal{ and }\\
J^N(T,u,\xi):=&\mathbb{E}\pp{\int_0^T\norm{\mathbf{Z}_s^{u,T,\xi}}^2\hat{q}\pr{ds,d\theta}+\frac{1}{N}\int_0^T\norm{u_s}^2ds},
\end{split}
\end{equation}for every admissible control $u\in\mathcal{U}$. Then, one easily proves that 
\begin{proposition}
The limit value function obtained by minimizing the cost functionals $J^n$ as $n$ increases to infinity coincides with $V$ (i.e. \begin{equation}
\inf_{N\geq 1}V^N(T,\xi)=V(T,\xi),\textnormal{ for all }\xi\in\mathbb{L}^2\left(\Omega,\mathcal{F}_{\pp{0,T}},\mathbb{P};\mathbb{R}^n\right).
\end{equation}
\end{proposition}
The proof is straightforward and will be omitted.
\subsection{Riccati Approach}
We consider the following equation
\begin{equation}\left\lbrace
\label{RBSDE_ M}
\begin{split}
d\Sigma_r^{M}=&\int_E\pp{\Sigma_r^{M}C_r^*(\theta)-\Theta_r^{M}(\theta)}\pr{I_n+\Sigma_r^{M}+\Theta_r^{M}(\theta)}^{-1}\pp{C_r(\theta)\Sigma_r^{M}-\Theta_r^{M}(\theta)}\hat{q}(drd\theta),\\
&+\int_E\Theta_r^{M}(\theta)\tilde{q}(drd\theta)+\pr{A_r\Sigma_r^{M}+\Sigma_r^{M}A_r^*-NB_rB_r^*}dr, \textnormal{ for all }0\leq r\leq T;\\
\Sigma_T^{M}=&M^{-1}I_n.
\end{split}
\right.
\end{equation}
These equations are understood either for $M\in\mathbb{N}$ or, with a slight abuse of notation, for $M=\infty$ (in which case, the final condition is understood to be the $0_n$ matrix). By solution of such equations we understand a couple of matrices \[\pr{\Sigma^{M},\Theta^{M}(\cdot)}\in \mathbb{L}^\infty\pr{\Omega;\mathbb{D}\pr{\pp{0,T};\mathcal{S}^n}}\times\mathbb{L}^\infty\pr{\Omega\times \pp{0,T}\times E,q;\mathcal{S}^n}\] satisfying 
\begin{itemize}
\item $\Sigma_t^{M}$ is positive definite if $M<\infty$ (and positive semi-definite if $M=\infty$) $\mathbb{P}\times \mathcal{L}eb$-a.s.,  $\Sigma_t^{M}+\Theta_t^{M}(\theta)$ is positive definite if $M<\infty$ (and positive semi-definite if $M=\infty$)$ \mathbb{P}(d\omega)\times \hat{q}(\omega, dt,d\theta)$- a.s.;
\item $\Sigma^{M}$ is continuous except, eventually, at jumping times;
\item $\mathbb{P}$-a.s., for all $0\leq T$, \begin{equation}
\label{RiccatiSigmaM}
\begin{split}
\Sigma_t^{M}=&M^{-1}I_n-\int_t^T\int_E\Theta_r^{M}(\theta)\tilde{q}(drd\theta)-\int_t^T\pr{A_r\Sigma_r^{M}+\Sigma_r^{M}A_r^*-NB_rB_r^*}dr\\
&-\int_t^T\int_E\pp{\Sigma_r^{M}C_r^*(\theta)-\Theta_r^{M}(\theta)}\pr{I_n+\Sigma_r^{M}+\Theta_r^{M}(\theta)}^{-1}\pp{C_r(\theta)\Sigma_r^{M}-\Theta_r^{M}(\theta)}\hat{q}(drd\theta).
\end{split}
\end{equation}
\end{itemize}
\begin{thm}
\label{ThMain}
Let us assume that the Riccati equation (\ref{RBSDE_ M}) admits a solution for some $M\in\mathbb{N}\cup\set{\infty}$. Then, for every $\xi\in\mathbb{L}^{2}\left(\Omega,\mathcal{F}_{\pp{0,T}},\mathbb{P};\mathbb{R}^{n}\right)$, the following assertions hold true.
\begin{itemize}
\item[1. ]The backward equation 
\begin{equation}
\label{EqEta}
\left\lbrace 
\begin{split}
d\eta_r^{M}=&A_r\eta_r^{M}dr-\int_E\pp{\Sigma_r^{M}C_r^*(\theta)-\Theta_r^{M}(\theta)}\zeta_r^{M}(\theta)\hat{q}(drd\theta)\\
&+\int_E\pr{C_r(\theta)\eta_r^M+\pr{I_n+\Sigma_r^{M}+\Theta_r^{M}(\theta)}\zeta_r^{M}(\theta)}\tilde{q}(drd\theta),\textnormal{ for all }0\leq r\leq T;\\
\eta_T^M=&-\xi.
\end{split}
\right.
\end{equation} admits a unique solution $\pr{\eta^M,\zeta^M}\in \mathbb{L}^2\pr{\Omega;\mathbb{D}\pr{\pp{0,T};\mathbb{R}^n}}\times\mathbb{L}^2\pr{\Omega\times \pp{0,T}\times E,q;\mathbb{R}^n}$ such that $\eta^M$ is continuous everywhere except, eventually, at jumping times.
\item[2. ]If $M\in\mathbb{N}$, then the value function satisfies \[V^N(T,\xi)\geq\mathbb{E}\pp{\int_0^T\int_E\scal{\pr{I+\Sigma_r^M+\Theta_r^M(\theta)}\zeta_r^M(\theta),\zeta_r^M(\theta)}\hat{q}(drd\theta)}.\]
\item[3.i.] If the the Riccati equation (\ref{RBSDE_ M}) admits a solution for every $M\in\mathbb{N}\cup\set{\infty}$, then
\[V^N(T,\xi)=\mathbb{E}\pp{\int_0^T\int_E\scal{\pr{I+\Sigma_r^\infty+\Theta_r^\infty(\theta)}\zeta_r^\infty(\theta),\zeta_r^\infty(\theta)}\hat{q}(drd\theta)}.\]
\item[3.ii.]The optimal control $u^*$ such that $V^N(T,\xi)=J^N\pr{T,\xi,u^*}$ is obtained by setting $u^*_t:=-NB_t^*Y_t$, where $Y$ is the unique $\mathbb{L}^2\pr{\Omega;\mathbb{D}\pr{\pp{0,T};\mathbb{R}^n}}$-solution  of
\begin{equation}
\label{Yoptimal}
\left\lbrace
\begin{split}
dY_t=&-A_t^*Y_t-\int_EC_t^*(\theta)\pr{\pr{I_n+\Sigma_t^\infty+\Theta_t^\infty(\theta)}^{-1}\pr{C_t(\theta)\Sigma_t^\infty-\Theta_t^\infty(\theta)}+\zeta_t^\infty(\theta)}\hat{q}(dtd\theta)\\
&+\int_E\pr{\pr{I_n+\Sigma_t^\infty+\Theta_t^\infty(\theta)}^{-1}\pr{C_t(\theta)\Sigma_t^\infty-\Theta_t^\infty(\theta)}+\zeta_t^\infty(\theta)}\tilde{q}(dtd\theta);\\
Y_0=&0;
\end{split}
\right.
\end{equation}
\end{itemize}
\end{thm} 
The proof will be postponed to Section \ref{SectionProofs}. Before giving the proof, one will need to understand some structural elements in the analysis of solutions of stochastic equations (especially in the backward setting). The equation (\ref{EqEta}) will turn out to be a standard linear backward stochastic differential equation and the solution (under a slightly more standard form) is guaranteed by \cite[Theorem 3.4]{Confortola_Fuhrman_2014}. The idea (and the principle of proof) for the lower bound on $V^N$ comes from the so-called "completion of squares" method. To obtain this, we introduce a companion process $Y^M$ and link the inner product $\scal{X^M+\eta^M,Y^M}$ to the Riccati equations (\ref{RiccatiDet}). Using some monotony and convergence properties for $\Sigma^M$ (obtained via structural arguments in Corollary \ref{CorStructuralPropSol}), we give an $M$-independent lower bound (the right-hand member of the equality in assertion 3.i.). Finally, exploiting the "optimal" companion process $Y$, the approach provides (as it is usual in LQ problems) the optimal control. We wish to emphasize that although the philosophy of the method is rather standard, we deal here with actual backward stochastic Riccati equations.\\
As we will see later on, the structural reduction of backward stochastic Riccati equations provides us with some ordinary differential systems of Riccati equations that are solvable. However, to come back to the (measurable) solutions of the initial problem, some measurability of selections has to be required. This can be avoided on discrete structures (and this justifies our Assumption \ref{AssEFinite} hereafter). Alternatively, stronger continuity for the coefficients has to be required in order to guarantee existence of measurable selections of matrix-valued solutions. We emphasize that this assumption is not needed to prove the previous result, nor the structural properties, but merely to guarantee proper measurability. 
\begin{assumption} \label{AssEFinite}
\begin{itemize}
\item[i.] The set $E$ is an at most countable family of modes endowed with the discrete topology.
\item[ii.] The coefficients $A,B,C$ are jump time-homogeneous (given a metric space $\mathcal{E}$, a predictable process $f:\Omega\times [0,T]\longrightarrow\mathcal{E}$ is said to be \textbf{jump time-homogeneous} if there exists a family of functions $f^j:E^{j+1}\times[0,T]\longrightarrow\mathcal{E}$ such that \[f_t(\omega)=f^j\pr{\gamma_0,\Gamma_{T_1}\pr{\omega},\ldots,\Gamma_{T_j}\pr{\omega},t},\] on $T_j(\omega)<t\leq T_{j+1}(\omega)$).
\end{itemize}
\end{assumption}
Then, Theorem \ref{ThMain} has the following particular formulation.
\begin{cor}
\label{corMain}
Whenever Assumption \ref{AssEFinite} holds true, then all the conclusions in the assertions [1-3]  in Theorem \ref{ThMain} are valid (without any further hypotheses).
\end{cor}The proof follows from Corollary \ref{CorStructuralPropSol} assertion 4.\\
In all the examples presented so far we have seen that either deterministic targets cannot be reached (in Example \ref{Exp1} or, for Examples \ref{Exp2}, \ref{Exp3}, every target can be reached. Computation of the value function $V$ allows  explicit identification of approximately reachable directions.
\begin{example}
\label{Exp4}
We consider a two-dimensional state space ($n=2$), a one-dimensional control space ($d=1$), a mode space $E:=\set{0,1}$, a transition intensity $\lambda=1$ and a post-jump measure $Q(\gamma)=\delta_{\set{1-\gamma}}$ for every $\gamma\in E$. Morover, we consider homogeneous coefficients given by \[A:=\begin{pmatrix}1 & 0 \\ 0 & 0 \end{pmatrix},  C=0_{2},  B:=\begin{pmatrix}0 \\ 1 \end{pmatrix}.\]  
\begin{itemize}
\item It is easy to see that the associated system (\ref{eq0}) is not approximately null-controllable starting from an(y)(non-zero) initial datum in $span\set{\begin{matrix}
1\\0
\end{matrix}}$.
\item Let us now turn to the Riccati system and the computation of $V^N$. We specify the dependence on $N$ of both the solution of the Riccati system and the couple $\pr{\eta^{N,\infty},\zeta^{N,\infty}}$ satisfying (\ref{EqEta}) with $M=\infty$. It is clear that the solution of our Riccati system is given by \[\Sigma_t^{N,\infty}=\pr{\begin{matrix}
0&0\\0&N(T-t)
\end{matrix}}, \Theta_\cdot^\infty=0_2.\]
For simplicity, we set  $\xi=\pr{\begin{matrix}
\xi^1\\\xi^2
\end{matrix}}$. By letting $\hat{\zeta}:=\pr{\begin{matrix}
\hat{\zeta}^1\\ \hat{\zeta}^2
\end{matrix}}$ be such that $\xi=\mathbb{E}\pp{\xi}+\int_0^T\hat{\zeta}_t(\theta)\tilde{q}(dtd\theta)$, it follows that 
$\zeta^{N,\infty}_t(\theta)=\pr{\begin{matrix}
e^{-t}\hat{\zeta}^1_t(\theta)\\ \frac{1}{1+N(T-t)}\hat{\zeta}^2_t(\theta)
\end{matrix}}$.
Hence, the value function is explicitly given as \[V^N\left(T,\begin{pmatrix}\xi^1 \\ \xi^2 \end{pmatrix}\right)=\mathbb{E}\pp{\int_{\pp{0,T}\times E}e^{-2t}\abs{\hat{\zeta}_t^1(\theta)}^2\hat{q}(dtd\theta)}+\mathbb{E}\pp{\int_{\pp{0,T}\times E}\frac{1}{1+N(T-t)}\abs{\hat{\zeta}^2_t(\theta)}^2\hat{q}(dtd\theta)}.\] Using dominated convergence and passing $N\rightarrow\infty$, one gets \[e^{-2T}\mathbb{E}\pp{\abs{\xi^1-\mathbb{E}\pp{\xi^1}}^2}\leq \inf_{N\geq 1}V^N\left(T,\begin{pmatrix}\xi^1 \\ \xi^2 \end{pmatrix}\right)\leq \mathbb{E}\pp{\abs{\xi^1-\mathbb{E}\pp{\xi^1}}^2}.\] As consequence, $\xi=\begin{pmatrix}\xi^1 \\ \xi^2 \end{pmatrix}$ satisfies $V\pr{T,\xi}=0$ (or, equivalently, is approximately reachable) if and only if $\xi^1$ is deterministic.
\end{itemize} 
\end{example}

\section{Structural Representation for Backward Stochastic Riccati Equations. Elements of Existence and Uniqueness}
\label{SectionStruct0}
\subsection{Structure Elements}\label{SectionStructure}
We will closely follow the ordinary differential approach in the study
of backward stochastic systems driven by marked point processes in
\cite{CFJ_2014}. The reader will recall that we have assumed that
the mode process is observed up to the $J$-th jump. We will introduce
a cemetery state $\pr{\infty,\Delta}$ to which the process is sent
after $\min\set{T_{J},T}$. We consider the space of elementary marks
$E_{T}:=\pp{0,T}\times E\cup\set{\pr{\infty,\Delta}}$. Next, we define
the space of all marks of length $(k+1)$ (basically corresponding
to the initial configuration $(0,\gamma_{0})$ to which we add the
first $J\geq k\geq0$ jump times and the corresponding post-jump positions):
\begin{equation}
E_{T}^{k+1}\supset E_{T,k}:=\left\lbrace \begin{split} & e=\pr{t_{0},\gamma_{0},t_{1},\gamma_{1},\ldots,t_{k},\gamma_{k}}\in E_{T}^{k+1}\textnormal{ such that }\\
 & t_{0}=0,\ \pr{t_{j}}_{0\leq j\leq k}\textnormal{ is a non-decreasing family; }t_{j}<t_{i+1},\textnormal{ if }t_{j}\leq T;\\
 & \pr{t_{j},\gamma_{j}}=\pr{\infty,\Delta},\textnormal{ if either }t_{j}>T,\textnormal{ and }0\leq j\leq\min\set{{k,J}-1}\textnormal{or }j=J.
\end{split}
\right\rbrace \label{ETk}
\end{equation}
This family will be endowed with its family of Borel sets denoted
$\mathcal{B}_{k}$. We introduce the notations 
\[
\abs e=k\textnormal{ such that }t_{\abs e}=t_{k},\textnormal{and }\gamma_{\abs e}:=\gamma_{k},\textnormal{ for such elements }e\in E_{T}^{k+1}.
\]
For $T\geq t>t_{\abs e}$ and $\theta\in E$, we set
the concatenation rule 
\[
e\oplus\pr{t,\theta}:=\pr{t_{0},\gamma_{0},\ldots,\abs{e},\gamma_{\abs{e}},t,\theta}\in E_{T}^{\abs{e}+1}.
\]
The reader will note that the $E_{T}^{k+1}$-valued random variable
given by $e_{k}:=\pr{0,\gamma_{0},T_{1},\Gamma_{T_{1}}^{\gamma_{0}},\ldots,T_{k},\Gamma_{T_{k}}^{\gamma_{0}}}$
corresponds to a mode trajectory. Let us now express \textit{measurable}
notions in this context.
\begin{itemize}
\item For the \textit{final datum} $\xi$ assumed to be $\mathcal{F}_{\left[0,T\right]}$
-measurable, this amounts to asking the existence of a family $\pr{\xi_{j}}_{0\leq j}$
such that $\xi_{j}:E_{T}^{j+1}\longrightarrow\mathbb{R}^n$ is
$\mathcal{B}_{j}\mid\mathcal{B}\pr{\mathbb{R}^{n}}$ -measurable such that
\[
\textnormal{If }\abs{e}=\infty,\textnormal{ then }\xi_{j}(e)=0.\textnormal{ On the set }T_{j}\leq T<T_{j+1},\textnormal{ one has } \xi(\omega)=\xi_{j}\pr{e_{j}\pr{\omega}}.
\]
\item The first component of the solution of our backward systems (either in the Euclidean space or on a matrix space) will consist of \textit{càdlàg
processes $\mathbf{{X}}$ continuous except at jumping times.} Such
processes (taking their values in some topological space $\mathcal{E}$ endowed with its Borel field $\mathcal{B}(\mathcal{E})$) are described by a family of $\mathcal{B}_{j}\otimes\mathcal{{B}}\pr{\pp{0,T}}/\mathcal{B}(\mathcal{E})$
-measurable functions $\mathbf{x}_{j}$ such that $\mathbf{x}_j\pr{e,\cdot}$ is continuous on $\pp{0,T}$ and constant on $\pp{0,\min\set{T,t_{\abs{e}}}}$ for all $e\in E_{T}^{j+1}$
\begin{equation*}
\textnormal{If }\abs{e}=\infty,\textnormal{then }\mathbf{x}_{j}(e,\cdot)=0.\textnormal{ On the set }T_{j}\leq t<T_{j+1},\textnormal{ one has } \mathbf{X}_t\pr{\omega}=\mathbf{x}_{j}\pr{e_{j}\pr{\omega},t},\textnormal{ for }t\leq T
\end{equation*}
\item The second component will consist of \textit{predictable processes }$\mathbf{Z}$. They are described by a family of $\mathcal{B}_{j}\otimes\mathcal{{B}}\pr{\pp{0,T}}\otimes\mathcal{B}(E)/\mathcal{B}(\mathcal{E})$-measurable functions $\mathbf{z}_{j}$ such that 
\begin{equation*}\begin{split}
&\textnormal{If }\abs{e}=\infty,\textnormal{then }\mathbf{z}_{j}(e,\cdot,\cdot)=0.\\
&\textnormal{On the set }T_{j}< t\leq T_{j+1},\textnormal{ one has } \mathbf{Z}_t\pr{\omega,\theta}=\mathbf{z}_{j}\pr{e_{j}\pr{\omega},t,\theta},\textnormal{ for }t\leq T,\theta\in E.
\end{split}
\end{equation*}
\item In this setting, \textit{the compensator} becomes \[\hat{q}\pr{\omega, dt, d\theta}=\sum_{j\geq0}\hat{q}_j\pr{e_j(\omega),dt,d\theta}\mathbbm{1}_{T_j(\omega)<t\leq \min\set{T_{j+1}(\omega),T}}\] such that 
\begin{equation*}
\begin{split}
&\textnormal{If }j\geq J,\textnormal{ then }\hat{q}_j\pr{e,dt,d\theta}=\delta_{\Delta}(d\theta)\delta_{\infty}(dt). \textnormal{ Otherwise,}\\ &\hat{q}_j\pr{e,dt,d\theta}:=\lambda\pr{\gamma_{\abs{e}},t}Q\pr{\gamma_{\abs{e}},t,d\theta}\mathbbm{1}_{t_{\abs{e}}<\infty,t_{\abs{e}}\leq t\leq T}dt+\delta_{\Delta}(d\theta)\delta_{\infty}(dt)\mathbbm{1}_{\set{t_{\abs{e}}<\infty,t>T}\cup \set{t_{\abs{e}}=\infty}}.
\end{split}
\end{equation*}
\item The reader will also recall that the coefficients are set to $0$ as soon as the cemetery state $\Delta$ is reached. Being \textit{predictable}, one identifies $A$ with a family of $\mathcal{B}_{j}\otimes\mathcal{{B}}\pr{\pp{0,T}}/\mathcal{B}(\mathbb{R}^{n\times n})$-measurable functions $A_j$ such that 
\begin{equation*}\begin{split}
&\textnormal{If }\abs{e}=\infty,\textnormal{then }A_{j}(e,\cdot)=0.\\
&\textnormal{On the set }T_{j}< t\leq	T_{j+1},\textnormal{ one has } A_t\pr{\omega}=A_{j}\pr{e_{j}\pr{\omega},t},\textnormal{ for }t\leq T,\theta\in E.
\end{split}
\end{equation*}
\item Similar assertions hold true for $B$ and $C$. 
\end{itemize}

\subsection{Structure Reduction to Deterministic Riccati Systems}
Let us fix, for the time being $M>0$ (we equally allow $M=\infty$). Inspired by \cite{CFJ_2014} (and using the structure elements described in Section \ref{SectionStructure}), we introduce the following system of iterated deterministic (backward) Riccati equations.
\begin{equation}
\label{RiccatiDet}
\left\lbrace
\begin{split}
\mathbf{\sigma}_J&\pr{e,\cdot}=M^{-1}I_n\\
\mathbf{\sigma}_j&\pr{e,T}=M^{-1}I_n,\textnormal{ for all }j\leq J-1,\ e\in E_T^{j+1}\\
d\mathbf{\sigma}_j&\pr{e,t}=\\
&\int_E \left[\left(
\begin{split}
&\pp{\mathbf{\sigma}_j\pr{e,t}\pr{C_j^*\pr{e,t,\theta}+I_n}-\mathbf{\sigma}_{j+1}\pr{e\oplus \pr{t,\theta},t}}\\
\times&\pp{I_n+\mathbf{\sigma}_{j+1}\pr{e\oplus \pr{t,\theta},t}}^{-1}\\
\times&\pp{\pr{C_j\pr{e,t,\theta}+I_n}\mathbf{\sigma}_j\pr{e,t}-\mathbf{\sigma}_{j+1}\pr{e\oplus \pr{t,\theta},t}}
\end{split}
\right)+\mathbf{\sigma}_j\pr{e,t}-\mathbf{\sigma}_{j+1}\pr{e\oplus \pr{t,\theta},t}\right]\hat{q}_j\pr{e,dt,d\theta}\\
&+\pp{A_j\pr{e,t}\mathbf{\sigma}_j\pr{e,t}+\mathbf{\sigma}_j\pr{e,t}A_j^*\pr{e,t}-NB_j\pr{e,t}B_j^*\pr{e,t}}dt\\
&\textnormal{ for all }j\leq J-1,\ e\in E_T^{j+1}\textnormal{ and all }t_{\abs{e}}\leq t\leq T,\\
\end{split}
\right.
\end{equation}
such that \begin{itemize}
\item for all $j\leq J$, $\mathbf{\sigma}_j\pr{e,\cdot}$ is positive definite if $M<\infty$ (resp. positive semi-definite for $M=\infty$);
\item for all $j\leq J$, $\mathbf{\sigma}_j\pr{e,\cdot}$ is continuous on $\pp{t_{\abs{e}},T}$;
\item for all $j\leq J$, $\mathbf{\sigma}_j\in \mathbb{L}^\infty\pr{E_T^{j+1}\times\pp{t_{\abs{e}},T};\mathcal{S}_+^n}$ (i.e. $\mathbf{\sigma}_j$ is $\mathcal{B}_{j+1}\otimes\mathcal{B}(\pp{0,T})\mid\mathcal{B}\pr{\mathcal{S}_+^n}$-measurable and bounded uniformly in $e\in E_T^{j+1}$).
\end{itemize}
\begin{remark}
\label{remarkMeasurability}The reader is kindly invited to note the following.
\begin{itemize}
\item[i.] The solutions can be considered global on $\pp{0,T}$ (instead of $\pp{t_{\abs{e}},T}$ by setting the coefficients to be $0$ prior to $t_{\abs{e}}$ and the solution to be constant on $\pp{0,t_{\abs{e}}}$.
\item [ii.] Since we have made the choice of setting to $0$ the coefficients after the $J$-th jump, $\mathbf{\sigma}_J\pr{e,\cdot}=M^{-1}I_n$ is equivalent to $\mathbf{\sigma}_{J+1}\pr{e,\cdot}=0$ and $\mathbf{\sigma}_J$ solving the same kind of system (with $\hat{q}_{J}$ a Dirac mass at $\infty$). This contributes to coherence with the structural definition of càdlàg solutions.
\item[iii.] When $M$ is allowed to vary, one should write $\sigma_j^M$ in the previous system (but the coefficients $A,B,C$ keep their independence).
\item[iv.] As we will see shortly after, generic equations appearing in (\ref{RiccatiDet}) are (always) solvable provided that a (correctly) measurable $\sigma_{j+1}$ is plugged in the equation at step $j$. Under the Assumption \ref{AssEFinite}, measurability of $\mathbf{\sigma}_j$ is always satisfied (as it only depends on the modes living in a discrete set and not on the jump times themselves).
\item[v.] Adding infinite activity (number of jumps) can be dealt with in a similar way. The interested reader can take a look a \cite{CFJ_2014} (where both cases are presented and the non-explosion-like condition introduced). We prefer to concentrate on this framework because it appears in gene networks (observed for a specified duration) and it is easier to implement the structure equation (as a finite system of ODE).
\end{itemize} 
\end{remark}
Using the structure elements in Section \ref{SectionStructure}, one gets (with the same proof as \cite[Lemma 7]{CFJ_2014}), the following.
\begin{proposition}
\label{PropReductionDetSystem}
For $M>0$ (respectively $M=\infty$), the system (\ref{RBSDE_ M}) is (uniquely) solvable and takes positive definite values (resp. positive semi-definite values) if and only if the family of (iterated deterministic Riccati backward) equations (\ref{RiccatiDet}) is (uniquely) solvable.
In this case, the solution is given by \begin{equation*}
\begin{split}
\Sigma_t^{M}(\omega)&=\mathbf{\sigma}_j\pr{e_j(\omega),t}, \textnormal{ if }T_j(\omega)\leq t< T_{j+1}(\omega),\\
\Theta_t^{M}(\omega,\theta)&=\mathbf{\sigma}_{j+1}\pr{e_j(\omega)\oplus\pr{t,\theta},t}-\mathbf{\sigma}_j\pr{e_j(\omega),t}, \textnormal{ if }T_j(\omega)< t\leq T_{j+1}(\omega).
\end{split}
\end{equation*}
\end{proposition}
\subsection{Solvability of Generic Equation Appearing in (\ref{RiccatiDet})}
For $j\leq J-1$ and $e\in E_T^{j+1}$ fixed, each equation of the system (\ref{RiccatiDet}) is of the form 
\begin{equation}
\label{GenericRiccatiM}
\left\lbrace
\begin{split}
dp_t(\sigma)&=p_t(\sigma)c_t(\sigma)p_t(\sigma)+a_t(\sigma)p_t(\sigma)+p_t(\sigma)a_t^*(\sigma)-b_t(\sigma),\textnormal{ for }t_{\abs{e}}< t\leq T,\\
p_T(\sigma)&=M^{-1}I_n.
\end{split}
\right.
\end{equation}The matrix coefficients are explicitly computed by setting, for $\sigma:=\sigma_{j+1}\in \mathbb{L}^\infty\pr{E_T^{j+2}\times\pp{0,T};\mathcal{S}_+^n}$,
\begin{equation}
\label{GenericRiccatiMCoeff}
\begin{split}
a_t\pr{\sigma}&:=A_j\pr{e,t}+\frac{1}{2}I_n-\int_E \mathbf{\sigma}\pr{e\oplus \pr{t,\theta},t}\pp{I_n+\mathbf{\sigma}\pr{e\oplus \pr{t,\theta},t}}^{-1}\pp{C_j\pr{e,t,\theta}+I_n}\lambda\pr{\gamma_{\abs{e}},t}Q\pr{\gamma_{\abs{e}},t,d\theta};\\
b_t\pr{\sigma}&:=NB_j\pr{e,t}B_j\pr{e,t}^*+\int_E \mathbf{\sigma}\pr{e\oplus \pr{t,\theta},t}\pp{I_n+\mathbf{\sigma}\pr{e\oplus \pr{t,\theta},t}}^{-1}\lambda\pr{\gamma_{\abs{e}},t}Q\pr{\gamma_{\abs{e}},t,d\theta};\\
c_t\pr{\sigma}&:=\int_E\pp{C_j^*\pr{e,t,\theta}+I_n}\pp{I_n+\mathbf{\sigma}\pr{e\oplus \pr{t,\theta},t}}^{-1}\pp{C_j\pr{e,t,\theta}+I_n}\lambda\pr{\gamma_{\abs{e}},t}Q\pr{\gamma_{\abs{e}},t,d\theta}.
\end{split}
\end{equation}
(The reader will note that the coefficients can be considered on $\pp{0,T}$ (e.g. $A_j(e,t)=0_n$ if $t<t_{\abs{e}}$ etc.) such that $p_t=p_{\max\set{t,t_{\abs{e}}}}$ in order to simplify notations).\\
We have the following result.
\begin{thm}
\label{TheoremSolvabilityDetSyst}
Let $\sigma \in \mathbb{L}^\infty\pr{E_T^{j+2}\times\pp{0,T};\mathcal{S}_+^n}$ be fixed. 
\begin{itemize}
\item[1.] The coefficients $a(\sigma)\in \mathbb{L}^\infty\pr{\pp{0,T};\mathbb{R}^{n\times n}}$, $b(\sigma)\in \mathbb{L}^\infty\pr{\pp{0,T};\mathcal{S}_+^{n}}$ and $c(\sigma)\in \mathbb{L}^\infty\pr{\pp{0,T};\mathcal{S}_+^{n}}$ and
\begin{equation}
\label{AssCoeffabc}
\norm{a_t}(\sigma)\leq \norm{A}_\infty+\frac{1}{2}+\norm{\lambda}_\infty \pr{\norm{C}_\infty+1},\ \norm{b_t}(\sigma)\leq N\norm{B}_\infty^2+\norm{\lambda}_\infty,\ \norm{c_t}(\sigma)\leq \pr{\norm{C}_\infty+1}^2.
\end{equation}
\item[2.] 
\begin{itemize}
\item[i.] The equation (\ref{GenericRiccatiM}) admits a positive semi-definite unique solution $p(\sigma)\in \mathbb{L}^\infty\pr{\pp{0,T};\mathcal{S}_+^n}$. 
\item[ii.] The  $\mathbb{L}^\infty$ norm of $p(\sigma)$ can be upper-bounded independently of $\sigma$ and $M\geq 1$.
\item[iii.] For $M<\infty$ fixed, there exists a positive constant $c_M>0$ (independent of $\sigma$) such that $p_t(\sigma)\geq c_MI_n$, for all $t\in[0,T]$.
\end{itemize}
\item[3.] Let us consider $\eta \in \mathbb{L}^\infty\pr{E_T^{j+2}\times\pp{0,T};\mathcal{S}_+^n}$ and $M'\in\mathbb{N}\cup\set{\infty}$. Moreover, we let $p(\eta)$ be the solution of the equation (\ref{GenericRiccatiM}) associated to $\eta$ such that $p_T(\eta)=\pr{M'}^{-1}I_n$ and $p(\sigma)$ be the solution of the equation (\ref{GenericRiccatiM}) associated to $\sigma$ such that $p_T(\sigma)=M^{-1}I_n$. If $\eta\leq\sigma$ and $M\leq M'$, then \[p_t(\eta)\leq p_t(\sigma),\textnormal{ for all } t_{\abs{e}}\leq t\leq T,\](inequalities are understood between matrices).
\end{itemize}
\end{thm}
(To ensure better readability), the proof is postponed to Section \ref{SectionProofs}. \\
As a consequence of the previous result, we get the following.
\begin {cor} 
\label{CorStructuralPropSol}
\begin{itemize} 
\item[1.] If the equation (\ref{RBSDE_ M}) admits a solution (for some $M\in\mathbb{N}\cup\set{\infty}$), then the solution is unique. This solution is uniformly upper-bounded and this upper-bound can be chosen independent of $M$.
\item[2.] If the equation (\ref{RBSDE_ M}) admits a solution (for some $M\in\mathbb{N}$), then $\Sigma^M$ and $\Sigma^M+\Theta^M\pr{\cdot}$ are positive definite and uniformly lower-bounded (by some $c_MI_n$ for some $c_M>0$). 
\item[3.] Let us assume that the equation (\ref{RBSDE_ M}) admits a solution for every $M\in\mathbb{N}\cup\set{\infty}$. Then  the unique solution $\pr{\Sigma^M,\Theta^M}$ converges uniformly (in $\mathbb{L}^\infty\pr{\Omega;\mathbb{D}\pr{\pp{0,T};\mathcal{S}^n}}\times\mathbb{L}^\infty\pr{\Omega\times \pp{0,T}\times E,q;\mathcal{S}^n}$) to $\pr{\Sigma^\infty,\Theta^\infty}$. Moreover, the sequences $\pr{\Sigma^M}_{M\geq 1}$ and $\pr{\Sigma^M+\Theta^M}_{M\geq 1}$ are non-increasing.
\item[4.] We assume \ref{AssEFinite} to hold true. Then the equation (\ref{RBSDE_ M}) admits a unique solution for every $M\in\mathbb{N}\cup\set{\infty}$. 
\end{itemize}
\end{cor}
Once again, the proof is postponed to Section \ref{SectionProofs}. Although the arguments follow (rather) immediately from the assertion of Theorem \ref{TheoremSolvabilityDetSyst}, we strive to provide our readers with the main elements.
\begin{remark}
In practice, given a direction of interest $\xi$, to check if the system can be driven $\varepsilon>0$-close to this direction, one proceeds as follows.
\begin{enumerate}
\item Pick $N$ and $M$ large enough.
\item Solve the backward Riccati structure system (\ref{RiccatiDet}) (leading to $\pr{\Sigma,\Theta}$).
\item Solve the (similar !) standard structure systems associated to (\ref{EqEta}) and (\ref{YrM}) (leading to $\pr{\eta,\zeta}$ resp. to $Y$).
\item Check that your desired error $\varepsilon$ does not exceed $\mathbb{E}\pp{\int_0^T\int_E\scal{\pr{I+\Sigma_r+\Theta_r(\theta)}\zeta_r(\theta),\zeta_r(\theta)}\hat{q}(drd\theta)}$.
\item If the error is acceptable with respect to $\varepsilon$ (meaning that the direction is approximately reachable), then the best control is $u^*_t:=-NB_t^*Y_t$.
\end{enumerate}
\end{remark}
\section{Conclusion and Perspectives}
In this paper we have given a controlled BSDE approach to the characterization of approximately reachable directions for switched systems presenting a controlled piecewise linear structure. The main results in Theorem \ref{ThMain} indicate that the controls allowing to (almost optimally) approach feasible targets (by solving the approximated problems $V^N$) have an intrinsic non-Markovian structure (a fortiori, they cannot be expected to be in a closed-loop form. Of course, for the initial problem with value function $V$, in general, the optimal control does not exist as approximately reachable targets can be envisaged without requiring them to be exactly reachable).\\
While the characterization given in our paper allows dealing with gene networks with single reactants (translating in linear systems), this class is hardly sufficient to investigate real-life systems (translating in polynomial dynamics). In a very early version, we have equally obtained some regularity properties for $V$ (suppressed due to their irrelevance for the mathematical arguments developed). Based on this regularity and versions of Peng's semigroup property for BSDEs, we hope to address the nonlinear case. This will provide an answer to general accurateness of models of gene networks (see discussions on bi-stability in \cite{crudu_debussche_radulescu_09}).
\section{Proofs of Results}
\label{SectionProofs}
We have gathered in this section the proofs of the previously stated results. Whenever the approach and the results are rather classical (for example, in connection to standard deterministic Riccati systems), only the main ingredients of the proofs are mentioned.\\
We begin with some elements of proof for the (rather standard) Proposition \ref{PropATC_abstract}]. We have hinted the connection between the dual of controllability operators and the BSDE (\ref{DualBSDE0}). Further elements can be found in \cite{GoreacMartinez2015}.
\begin{proof}[Proof of Proposition \ref{PropATC_abstract}] One simply notes that approximate terminal controllable to $\xi$ is equivalent to the closure of the range of the operator $L$ in $\mathbb{L}^2\left(\Omega,\mathcal{F}_{\pp{0,T}},\mathbb{P};\mathbb{R}^n\right)$ containing the range of the operator $L':\mathbb{R}\rightarrow\mathbb{L}^2\left(\Omega,\mathcal{F}_{\pp{0,T}},\mathbb{P};\mathbb{R}^n\right)$ given by $L'(r):=r\xi,$ for all $r\in \mathbb{R}$. Equivalently, one  can write the condition on the kernels of the adjoint operators i.e. $\ker\pr{L^*}\subset\ker\pr{\pr{L'}^*}$. The identification of this dual allows to complete the proof.
\end{proof}
We now turn to the structural equation (\ref{RiccatiDet}) and the proof of Theorem \ref{TheoremSolvabilityDetSyst}. Most of the assertions follow from the standard form of the generic equation (\ref{GenericRiccatiM}) and, for such assertions, we only sketch the arguments. 
\begin{proof}[Proof of Theorem \ref{TheoremSolvabilityDetSyst}]
\begin{itemize}
\item[1.] The first assertion is a simple consequence of the Assumption \ref{AssCoeff}. The reader will note that these bounds do not depend on the iterations. \\
\item[2.i.] In order to prove the second assumption, let us simplify the notation by dropping the dependence on $\sigma$. The reader is invited to note that $c$ and $b$ are positive semi-definite (by definition). Once the class of coefficients established, the reader is invited to note that the equation (\ref{GenericRiccatiM}) is now a standard one (associated to deterministic LQ problems). We will only sketch the proof.
Existence and uniqueness of a  solution follows from classical arguments on deterministic Riccati equations (e.g. \cite[Chapter 6, Corollary 2.10]{yong_zhou_99}) if $c_t\gg0$, for all $0\leq t\leq T$. Otherwise, one uses classic tricks by considering the penalized coefficients $c_t^{k}:=c_t+k^{-1} I_n$ and the associated solution $p^k$ satisfying the equation \begin{equation*}
\left\lbrace
\begin{split}
dp_t^k&=p_t^kc_t^kp_t^k+a_tp_t^k+p_t^ka_t^*-b_t,\textnormal{ for }t_{\abs{e}}< t\leq T,\\
p_T^k&=M^{-1}I_n.
\end{split}
\right.
\end{equation*} Then,
\begin{itemize}
\item the solutions $p^k$ are uniformly bounded from above by the solution of the Lyapunov equation (obtained by merely taking $c^k=0$).
\item the sequence $\pr{p^k}_k\in \mathbb{L}^\infty\pr{\pp{0,T};\mathbb{S}_+^n}$ is nondecreasing in $k$. This is obtained by writing down the standard Riccati equation satisfied by $p^k-p^{k'}$;
\item the conclusion follows by taking the limit as $k\uparrow\infty$ in the integral form of the solution $p_k$ \[\pr{\textnormal{i.e. }p_t^k=M^{-1}I_n-\int_t^T\pr{p_s^kc_s^kp_s^k+a_sp_s^k+p_s^ka_s^*-b_s}ds}\] and invoking dominated convergence.
\item The reader is invited to note that, for any integer $M\geq 1$, the solution $p$ is actually positive definite (a glance at 2. iii. will actually provide the reader with the inverse of such matrices). The same kind of arguments work for $M=\infty$ (i.e. when the final condition is $p_T=0_n$) except, in this case, the solution is only positive semi-definite.
\end{itemize}
\item[2.ii.] One only needs to note that all solutions $p^k$ are bounded from above by the solution of the Lyapunov equation \[\left\lbrace
\begin{split}
q_T&=I_n,\\
dq_t&=a_tq_t+q_ta_t^*-b_t,\textnormal{ for }t_{\abs{e}}< t\leq T.
\end{split}\right.
\]which is uniformly bounded owing to Gronwall's inequality and (\ref{AssCoeffabc}).
\item[2.iii.] Let us consider $M>0$ (finite and) fixed. Then the inverse of the solution $p(\sigma)$ given by equation (\ref{GenericRiccatiM}) satisfies 
\begin{equation*}
\left\lbrace
\begin{split}
dp_t^{-1}(\sigma)&=p_t^{-1}(\sigma)b_t(\sigma)p_t^{-1}(\sigma)-a_t^*(\sigma)p_t^{-1}(\sigma)-p_t^{-1}(\sigma)a_t(\sigma)-c_t(\sigma),\textnormal{ for }t_{\abs{e}}< t\leq T,\\
p_T^{-1}(\sigma)&=MI_n.
\end{split}
\right.
\end{equation*}
Existence and uniqueness for the solution of such equations follow the same arguments in 2. i. Moreover, owing to (\ref{AssCoeffabc}), it follows that $p^{-1}$ is upper bounded by some $c_M^{-1}I_n$ (independent of $\sigma$ but possibly depending on $M$). Our assertion follows.
\item[3.] To prove the third assertion, let us begin with simplifying notation and set $p^1:=p(\sigma)$ and $p^2:=p(\eta)$. Then
\begin{equation*}
\left\lbrace
\begin{split}
d\pr{p^1-p^2}_t&=\pr{p^1-p^2}_tc_t(\sigma)\pr{p^1-p^2}_t+\hat{a}_t\pr{p^1-p^2}_t+\pr{p^1-p^2}_t\hat{a}_t^*-\hat{c}_t,\textnormal{ for }t_{\abs{e}}\leq t\leq T,\\
d\pr{p^1-p^2}_T&=\pr{M^{-1}-\pr{M'}^{-1}}I_n.
\end{split}
\right.
\end{equation*}
Here, $\hat{a}_t=a_t(\sigma)+p_t^2c_t(\sigma)$ and \[\hat{c}_t=p_t^2\pp{-c_t(\sigma)+c_t(\eta)}p_t^2+\pp{a_t(\eta)-a_t(\sigma)}p_t^2+p_t^2\pp{a_t(\eta)-a_t(\sigma)}^*+b_t(\sigma)-b_t(\eta).\] We only need to prove that $\hat{c}_t\geq 0_n$ and the conclusion will follow (the previous equation becomes a standard Riccati one having a positive semi-definite solution). \\
Owing to (\ref{GenericRiccatiMCoeff}), one gets
\begin{equation}
\begin{split}
\hat{c}_t=&-p_t^2\pp{\int_E\pp{C_j^*\pr{e,t,\theta}+I_n}\pp{I_n+\mathbf{\sigma}\pr{e\oplus \pr{t,\theta},t}}^{-1}\pp{C_j\pr{e,t,\theta}+I_n}\lambda\pr{\gamma_{\abs{e}},t}Q\pr{\gamma_{\abs{e}},t,d\theta}}p_t^2\\
&+p_t^2\pp{\int_E\pp{C_j^*\pr{e,t,\theta}+I_n}\pp{I_n+\mathbf{\eta}\pr{e\oplus \pr{t,\theta},t}}^{-1}\pp{C_j\pr{e,t,\theta}+I_n}\lambda\pr{\gamma_{\abs{e}},t}Q\pr{\gamma_{\abs{e}},t,d\theta}}p_t^2\\
&+\int_E \pp{I_n+\mathbf{\sigma}\pr{e\oplus \pr{t,\theta},t}}^{-1}\pp{C_j\pr{e,t,\theta}+I_n}\lambda\pr{\gamma_{\abs{e}},t}Q\pr{\gamma_{\abs{e}},t,d\theta}p_t^2\\
&-\int_E \pp{I_n+\mathbf{\eta}\pr{e\oplus \pr{t,\theta},t}}^{-1}\pp{C_j\pr{e,t,\theta}+I_n}\lambda\pr{\gamma_{\abs{e}},t}Q\pr{\gamma_{\abs{e}},t,d\theta}p_t^2\\
&+\pr{\int_E \pp{I_n+\mathbf{\sigma}\pr{e\oplus \pr{t,\theta},t}}^{-1}\pp{C_j\pr{e,t,\theta}+I_n}\lambda\pr{\gamma_{\abs{e}},t}Q\pr{\gamma_{\abs{e}},t,d\theta}p_t^2}^*\\
&-\pr{\int_E \pp{I_n+\mathbf{\eta}\pr{e\oplus \pr{t,\theta},t}}^{-1}\pp{C_j\pr{e,t,\theta}+I_n}\lambda\pr{\gamma_{\abs{e}},t}Q\pr{\gamma_{\abs{e}},t,d\theta}p_t^2}^*\\
&-\int_E \pp{I_n+\mathbf{\sigma}\pr{e\oplus \pr{t,\theta},t}}^{-1}\lambda\pr{\gamma_{\abs{e}},t}Q\pr{\gamma_{\abs{e}},t,d\theta}\\
&+\int_E \pp{I_n+\mathbf{\eta}\pr{e\oplus \pr{t,\theta},t}}^{-1}\lambda\pr{\gamma_{\abs{e}},t}Q\pr{\gamma_{\abs{e}},t,d\theta}.
\end{split}
\end{equation}
By rewriting the terms and owing to the assumption (that $\eta\leq\sigma$), one gets
\end{itemize}
\begin{equation}
\hat{c}_t\geq 
\int_E \left[
\begin{split}
\pr{\pp{C_j\pr{e,t,\theta}+I_n}p_t^2-I_n}^*\\
\times\pr{\pp{I_n+\mathbf{\eta}\pr{e\oplus \pr{t,\theta},t}}^{-1}
-\pp{I_n+\mathbf{\sigma}\pr{e\oplus \pr{t,\theta},t}}^{-1}}\\
\times\pr{\pp{C_j\pr{e,t,\theta}+I_n}p_t^2-I_n}\end{split}\right]\lambda\pr{\gamma_{\abs{e}},t}Q\pr{\gamma_{\abs{e}},t,d\theta}
\geq 0_n.
\end{equation}This completes our proof.
\end{proof}
Finally, for Corollary \ref{CorStructuralPropSol} we give the following.
\begin{proof}[Sketch of the proof of Corollary \ref{CorStructuralPropSol}]
\begin{itemize}
\item[1.] Uniqueness in first assertion follows from Proposition \ref{PropReductionDetSystem} and the uniqueness part in Theorem \ref{TheoremSolvabilityDetSyst}, assertion 2 i. by descending recurrence over $j\leq J$. The (uniform) upper-bounds are provided by Theorem \ref{TheoremSolvabilityDetSyst}, assertion 2 ii.
\item[2.] Owing to the representation of the solution in Proposition \ref{PropReductionDetSystem}, both the solution $\Sigma^M$ and $\Sigma^M+\Theta^M\pr{\cdot}$ can be represented as solutions to equations of form (\ref{GenericRiccatiM}). Then, lower-bound follows from Theorem \ref{TheoremSolvabilityDetSyst}, assertion 2. iii.
\item[3.] For the convergence, one proceeds as follows. 
\begin{itemize}
\item If $M'\geq M$, one proves by descending recurrence over $j\leq J$ and owing to Theorem \ref{TheoremSolvabilityDetSyst} assertion 3. that \[\sigma^{M'}\pr{e,\cdot}\leq \sigma^M\pr{e,\cdot},\textnormal{ for all }e\in E_T^{j+1}.\]
\item One deduces the existence of $\sigma\pr{e,\cdot}$ the (possibly discontinuous) limit as $M\uparrow\infty$ of $\sigma^{M}\pr{e,\cdot}$. Identification of $\sigma\pr{e,\cdot}$ and $\sigma^\infty\pr{e,\cdot}$ is done using the integral form of the solution and dominated convergence. Domination is guaranteed by the first assertion in our Corollary (note that the uniform in $M$ upper-bound is valid both for $\Sigma^M$ and $\Sigma^M+\Theta^M(\cdot)$).
\end{itemize}
\item[4.] Assumption \ref{AssEFinite} is only needed in order to guarantee measurability of $\sigma_{j}$ for $j\leq J-1$ (see last item in Remark \ref{remarkMeasurability}). Under this assumption, the (unique) solution is constructed by (descending) recurrence (over $j$) by using Theorem \ref{TheoremSolvabilityDetSyst} assertion 2. One concludes owing to the representation Proposition \ref{PropReductionDetSystem}.
\end{itemize}
\end{proof}
Having established the structural properties, we are now able to complete the proof of the reachability-linked Theorem \ref{ThMain}.
\begin{proof}[Proof of Theorem \ref{ThMain}]
1. To prove the first assertion, we begin with introducing the following (more standard) backward stochastic differential equation.
\begin{equation}
\label{EqEtaModif}
\left\lbrace 
\begin{split}
d\bar{\eta}_r^{M}=&\int_E\bar{\zeta}_r^{M}(\theta)\tilde{q}(drd\theta)+A_r\bar{\eta}_r^{M}dr+\pr{\Sigma_r^{M}C_r^*(\theta)-\Theta_r^{M}(\theta)}\pr{I_n+\Sigma_r^{M}+\Theta_r^{M}(\theta)}^{-1}C_r(\theta)\bar{\eta}_r^{M}\hat{q}(drd\theta)\\
&-\int_E\pp{\Sigma_r^{M}C_r^*(\theta)-\Theta_r^{M}(\theta)}\pr{I_n+\Sigma_r^{M}+\Theta_r^{M}(\theta)}^{-1}\bar{\zeta}_r^{M}(\theta)\hat{q}(drd\theta),\textnormal{ for all }0\leq r\leq T;\\
\eta_T^M=&-\xi.
\end{split}
\right.
\end{equation}
It is then clear that, provided that existence and uniqueness for either solution holds true, one toggles between the two equations by setting $\pr{\bar{\eta}_r^{M},\bar{\zeta}_r^{M}(\theta)}=\pr{\eta_r^M,C_r(\theta)\eta_r^M+\pr{I_n+\Sigma_r^{M}+\Theta_r^{M}(\theta)}\zeta_r^M(\theta)}$. Keeping adequate integrability is guaranteed by Corollary \ref{CorStructuralPropSol} and the assumption on the coefficients (Assumption \ref{AssCoeff}).\\
Owing to Corollary \ref{CorStructuralPropSol} assertion 1. and to Assumption \ref{AssCoeff}, it follows that the equation (\ref{EqEtaModif}) is a standard linear backward stochastic differential equation whose coefficients are (essentially) bounded. Existence and uniqueness follow from standard arguments (e.g. \cite[Theorem 3.4]{Confortola_Fuhrman_2014}, see also \cite{CFJ_2014} for further reduction to systems of ordinary differential equations).\\
We now prove the remaining assertions. We proceed in several steps and begin with fixing (for the time being) $M\in\mathbb{N}$.\\
\underline{Step 0.} (Control interpretation for the BSDE (\ref{BSDE1}))\\
As it is often the case, the $\mathbf{Z}$ component in the equation (\ref{BSDE1}) can be interpreted as a predictable control and the solution $\mathbf{X}$ as a controlled, forward one by setting
\begin{equation}
\left\lbrace
\begin{split}
d\mathbb{X}_t^{x,u,\mathbb{Z}}&=\pr{A_t\mathbb{X}_t^{x,u,\mathbb{Z}}+B_tu_t}dt+\int_E\pr{C_t(\theta)\mathbb{X}_t^{x,u,\mathbb{Z}}+\mathbb{Z}_t(\theta)}\tilde{q}(dt,d\theta),\ 0\leq t\leq T,\\
\mathbb{X}_0^{x,u,\mathbb{Z}}&=x\in\mathbb{R}^n.
\end{split}
\right.
\end{equation}
 Then, the value function $V^N$ is the infimum over such $\pr{x,u,\mathbb{Z}}$ of the cost functional \[\mathbb{J}^N\pr{x,u,\mathbb{Z},T}:=\frac{1}{N}\mathbb{E}\pp{\int_0^T\norm{u_r}^2dr}+\mathbb{E}\pp{\int_0^T\int_E\norm{\mathbb{Z}_r(\theta)}^2\hat{q}(drd\theta}\] computed for trajectories satisfying the final constraint $\mathbb{X}_T^{x,u,\mathbb{Z}}=\xi$. The reader is invited to note that this implies \begin{equation}\label{LowerBddVN}
 \begin{split}
 &V^N(T,\xi)\geq \\
 &\inf_{x\in\mathbb{R}^n,u\in\mathcal{U},\mathbb{Z}\in\mathbb{L}^2\pr{\Omega\times [0,T]\times E,\mathcal{P},q;\mathbb{R}^n}}M\mathbb{E}\norm{\mathbb{X}_T^{x,u,\mathbb{Z}}-\xi}^2+\frac{1}{N}\mathbb{E}\pp{\int_0^T\norm{u_r}^2dr}+\mathbb{E}\pp{\int_0^T\int_E\norm{\mathbb{Z}_r^{T,\xi,u}(\theta)}^2\hat{q}(drd\theta}.
 \end{split}
\end{equation}
\underline{Step 1.} (Lower bound)\\ 
At this point, we consider, for $x\in\mathbb{R}^n$, $u\in\mathcal{U}$ and $\mathbb{Z}\in\mathbb{L}^2\pr{\Omega\times [0,T]\times E,\mathcal{P},q;\mathbb{R}^n}$ predictable fixed, the following (forward stochastic differential system).
\begin{equation}
\label{YrM}
\begin{split}
&\left\lbrace 
\begin{split}
dY_r^{M}=&\pr{-A_r^*Y_r^{M}+N\pr{\Sigma_r^{M}}^{-1}B_rB_r^*Y_r^{M}+\pr{\Sigma_r^{M}}^{-1}B_ru_r}dr\\
&\int_E \left(\begin{split}
&\pp{\pr{\Sigma_r^{M}+\Theta_r^{M}(\theta)}^{-1}-\pr{\Sigma_r^{M}}^{-1}}\mathbb{Z}_r(\theta)\\
+&\pp{\pr{\Sigma_r^{M}+\Theta_r^{M}(\theta)}^{-1}-\pr{\Sigma_r^{M}}^{-1}-C_r^*(\theta)}\times\\
&\times\pr{I_n+\Sigma_r^{M}+\Theta_r^M(\theta)}^{-1} \pp{C_r(\theta)\Sigma_r^{M}-\Theta_r^{M}(\theta)}Y_r^M\\
+&\pp{\pr{\Sigma_r^{M}+\Theta_r^{M}(\theta)}^{-1}-\pr{\Sigma_r^{M}}^{-1}-C_r^*(\theta)}\zeta_r^M(\theta)
\end{split}\right) \hat{q}(drd\theta)\\
&+\int_E \left(\begin{split} &\pr{\Sigma_r^{M}+\Theta_r^M(\theta)}^{-1}\mathbb{Z}_r(\theta)\\
+&\pr{\Sigma_r^{M}+\Theta_r^M(\theta)}^{-1} \pp{C_r(\theta)\Sigma_r^{M}-\Theta_r^{M}(\theta)}Y_r^M\\
+&\pp{I+\pr{\Sigma_r^{M}+\Theta_r^M(\theta)}^{-1}}\zeta_r^{M}(\theta)\end{split}\right) \tilde{q}(drd\theta),\\
Y_0^M=&\pr{\Sigma_0^M}^{-1}\pr{X_0+\eta_0^M}.
\end{split}
\right.
\end{split}
\end{equation}
Existence and uniqueness of this solution is standard (the reader only needs to recall Corollary \ref{CorStructuralPropSol} assertion 2. and Assumption \ref{AssCoeff}). The reader may want to note (using, for example, Itô's formula applied to $\Sigma_r^MY_r^M$ on $[0,t]$) that
\begin{equation}
\label{trajectory_eq}
\Sigma_r^MY_r^M=\mathbb{X}_r^{x,u,\mathbb{Z}}+\eta_r^M.
\end{equation}
We emphasize (and this is even more obvious according to this last equality) that $Y^M$ actually also depends on $x,u,\mathbb{Z}$. We have dropped this dependence in order to simplify the notations.\\
\underline{Step2.} (Link with the optimization problem)\\
By applying Itô's formula to the Euclidean product between $\mathbb{X}_r^{x,u\mathbb{Z}}+\eta_r^M$ and $Y_r^M$ on $[0,T]$, one establishes the equality
\begin{equation}
\begin{split}
&\pr{M\mathbb{E}\norm{X_T^{x,u,\mathbb{Z}}-\xi}^2+\frac{1}{N}\mathbb{E}\pp{\int_0^T\norm{u_r}^2dr}+\mathbb{E}\pp{\int_0^T\int_E\norm{\mathbb{Z}_r^{T,\xi,u}(\theta)}^2\hat{q}(drd\theta}}\\
=&\mathbb{E}\pp{\scal{Y_0^{M},x+\eta_0^M}}\\
+&\mathbb{E}\pp{\int_0^t\int_E\norm{\begin{split}&\pr{I+\pr{\Sigma_r^M+\Theta_r^M(\theta)}^{-1}}^\frac{-1}{2}\pr{\Sigma_r^M+\Theta_r^M(\theta)}^{-1}\pr{C_r(\theta)\Sigma_r^M-\Theta_r^M(\theta)}Y_r^{M}\\
+&\pr{I+\pr{\Sigma_r^M+\Theta_r^M(\theta)}^{-1}}^\frac{1}{2}\zeta_r^M+\pr{I+\pr{\Sigma_r^M+\Theta_r^M(\theta)}^{-1}}^\frac{1}{2}\mathbb{Z}_r^{T,\xi,u}(\theta)\end{split}}^2\hat{q}(drd\theta)}\\
&+\mathbb{E}\pp{\int_0^T\norm{\sqrt{N}B_r^*Y_r^{M}+\frac{1}{\sqrt{n}}u_r}^2dr}\\
&+\mathbb{E}\pp{\int_0^T\int_E\scal{\pr{I+\Sigma_r^M+\Theta_r^M(\theta)}\zeta_r^M(\theta),\zeta_r^M(\theta)}\hat{q}(drd\theta)}
\end{split}
\end{equation}
As consequence, using (\ref{LowerBddVN}), one gets the (announced) lower bound for the (penalized) cost function
\begin{equation}
V^N(T,\xi)\geq\mathbb{E}\pp{\int_0^T\int_E\scal{\pr{I+\Sigma_r^M+\Theta_r^M(\theta)}\zeta_r^M(\theta),\zeta_r^M(\theta)}\hat{q}(drd\theta)}.
\end{equation}
\underline{Step 3.} (Limit)\\
A simple glance at Corollary \ref{CorStructuralPropSol} assertion 3. shows that $I_n+\Sigma^\infty+\Theta^\infty\leq I_n+\Sigma^M+\Theta^M$ for every $M\geq 1$. It follows that \begin{equation}\label{lowerbdVN}V^N(T,\xi)\geq\mathbb{E}\pp{\int_0^T\int_E\scal{\pr{I+\Sigma_r^\infty+\Theta_r^\infty(\theta)}\zeta_r^M(\theta),\zeta_r^M(\theta)}\hat{q}(drd\theta)}.\end{equation}
Second, owing to Corollary \ref{CorStructuralPropSol} assertion 3. and the (uniform) upper-bounds in Corollary \ref{CorStructuralPropSol} assertion 1., respectively Assumption \ref{AssCoeff}, it follows that the coefficients of the linear equation (\ref{EqEta}) written for $M\in\mathbb{N}$ i.e.
$\pr{\Sigma_r^{M}C_r^*(\theta)-\Theta_r^{M}(\theta)}$ and
$\pr{I_n+\Sigma_r^{M}+\Theta_r^{M}(\theta)}^{-1}$ converge in $\mathbb{L}^\infty\pr{\Omega\times\pp{0,T}\times E,\mathcal{P},\hat{q};\mathbb{R}^n}$ to those of equation (\ref{EqEta}) with $M=\infty$. Using linearity and classical estimates for BSDE (e.g. \cite[Corollary 3.6]{Confortola_Fuhrman_2014}), one establishes the convergence of $\zeta^M$ to $\zeta^\infty$ in $\mathbb{L}^2\pr{\Omega\times\pp{0,T}\times E,\mathcal{P},q;\mathbb{R}^n}$. By allowing $M\rightarrow\infty$ in (\ref{lowerbdVN}), one gets\begin{equation*}V^N(T,\xi)\geq\mathbb{E}\pp{\int_0^T\int_E\scal{\pr{I+\Sigma_r^\infty+\Theta_r^\infty(\theta)}\zeta_r^\infty(\theta),\zeta_r^\infty(\theta)}\hat{q}(drd\theta)}.\end{equation*}\\
\underline{Step 4.} (Optimality)\\
The final task is to prove that the previous inequality is actually an equality. Starting from $Y$ defined in (\ref{Yoptimal}) (well-posedness following again from Corollary \ref{CorStructuralPropSol} assertion 1.), we set $X:=\Sigma^\infty Y-\eta^\infty$. This process satisfies
\begin{equation}
\left\lbrace
\begin{split}
dX_t=&\pr{A_tX_t-NB_tB_t^*Y_t}dt\\
&+\int_E\set{C_t(\theta)X_t-\pp{\pr{I_n+\Sigma_t^\infty+\Theta_t^\infty(\theta)}^{-1}\pr{C_t(\theta)\Sigma_t^\infty-\Theta_t^\infty(\theta)}Y_t+\zeta_t^\infty(\theta)}}\tilde{q}(dtd\theta)\\
X_T=&\xi.
\end{split}
\right.
\end{equation}
One deduces that $X$ can be identified with $\mathbf{X}^{T,\xi,u^*}$, for $u_r^*=-NB_r^*Y_r$. Moreover, it follows by uniqueness that $\mathbf{Z}_t^{T,\xi,u^*}(\theta)=-\pr{\pr{I_n+\Sigma_t^\infty+\Theta_t^\infty(\theta)}^{-1}\pr{C_t(\theta)\Sigma_t^\infty-\Theta_t^\infty(\theta)}Y_t+\zeta_t^\infty(\theta)}$.
To compute the cost functional $J^N\pr{T,\xi,u^*}$, one notes that $X_0+\eta_0^\infty=0=X_T+\eta_T^\infty$. Then, by simply applying Itô's formula to $\scal{Y,X+\eta^\infty}$ on $\pp{0,T}$, one gets \[J^N\pr{T,\xi,u^*}=\mathbb{E}\pp{\int_0^T\int_E\scal{\pr{I+\Sigma_r^\infty+\Theta_r^\infty(\theta)}\zeta_r^\infty(\theta),\zeta_r^\infty(\theta)}\hat{q}(drd\theta)}.\]The proof of our result is now complete.
\end{proof}

\bibliographystyle{plain}
\bibliography{../LatestBib/bibliografie_2018.bib} 

\def\cprime{$'$}
\begin{thebibliography}{10}

\bibitem{Barbu_Rascanu_Tessitore_2003}
V.~Barbu, A.~R{\u{a}}{\c{s}}canu, and G.~Tessitore.
\newblock Carleman estimates and controllability of linear stochastic heat
  equations.
\newblock {\em Appl. Math. Optim.}, 47(2):97--120, 2003.

\bibitem{Bismut_73}
J.-M. Bismut.
\newblock Conjugate convex functions in optimal stochastic control.
\newblock {\em J. Math. Anal. Appl.}, 44:384--404, 1973.

\bibitem{Bremaud_1981}
P.~Br{\'e}maud.
\newblock {\em Point processes and queues : martingale dynamics}.
\newblock Springer series in statistics. Springer-Verlag, New York, 1981.

\bibitem{Buckdahn_Quincampoix_Tessitore_2006}
R.~Buckdahn, M.~Quincampoix, and G.~Tessitore.
\newblock A characterization of approximately controllable linear stochastic
  differential equations.
\newblock In {\em Stochastic partial differential equations and
  applications---{VII}}, volume 245 of {\em Lect. Notes Pure Appl. Math.},
  pages 53--60. Chapman \& Hall/CRC, Boca Raton, FL, 2006.

\bibitem{Confortola_Fuhrman_2014}
F.~Confortola and M.~Fuhrman.
\newblock Backward stochastic differential equations associated to jump
  {M}arkov processes and applications.
\newblock {\em Stochastic Process. Appl.}, 124(1):289--316, 2014.

\bibitem{CFGT_2018}
F.~Confortola, M.~Fuhrman, G.~Guatteri, and G.~Tessitore.
\newblock Linear-quadratic optimal control under non-markovian switching.
\newblock {\em Stochastic Analysis and Applications}, 36(1):166--180, 2018.

\bibitem{CFJ_2014}
F.~Confortola, M.~Fuhrman, and J.~Jacod.
\newblock Backward stochastic differential equation driven by a marked point
  process: An elementary approach with an application to optimal control.
\newblock {\em Ann. Appl. Probab.}, 26(3):1743--1773, 06 2016.

\bibitem{crudu_debussche_radulescu_09}
A.~Crudu, A.~Debussche, and O.~Radulescu.
\newblock Hybrid stochastic simplifications for multiscale gene networks.
\newblock {\em BMC Systems Biology}, page 3:89, 2009.

\bibitem{Curtain_86}
R.~F. Curtain.
\newblock {Invariance concepts in infinite dimensions}.
\newblock {\em {SIAM J. Control and Optim.}}, {24}({5}):{1009--1030}, {SEP}
  {1986}.

\bibitem{Davis_84}
M.~H.~A. Davis.
\newblock {Piecewise-deterministic Markov-processes - A general-class of
  non-diffusion stochastic-models}.
\newblock {\em {Journal of the Royal Statistical Society Series
  B-Methodological}}, {46}({3}):{353--388}, {1984}.

\bibitem{davis_93}
M.~H.~A. Davis.
\newblock {\em {M}arkov models and optimization}, volume~49 of {\em Monographs
  on Statistics and Applied Probability}.
\newblock Chapman \& Hall, London, 1993.

\bibitem{Fernandez_Cara_Garrido_atienza_99}
E.~Fern{\'a}ndez-Cara, M.~J. Garrido-Atienza, and J.~Real.
\newblock On the approximate controllability of a stochastic parabolic equation
  with a multiplicative noise.
\newblock {\em C. R. Acad. Sci. Paris S\'er. I Math.}, 328(8):675--680, 1999.

\bibitem{G17}
D.~Goreac.
\newblock A {K}alman-type condition for stochastic approximate controllability.
\newblock {\em Comptes Rendus Mathematique}, 346(3--4):183 -- 188, 2008.

\bibitem{GoreacGrosuRotenstein_2016}
D.~Goreac, A.~C. Grosu, and E.-P. Rotenstein.
\newblock Approximate and approximate null-controllability of a class of
  piecewise linear markov switch systems.
\newblock {\em {Systems \& Control Letters}}, 96:118 -- 123, 2016.

\bibitem{GoreacMartinez2015}
D.~Goreac and M.~Martinez.
\newblock Algebraic invariance conditions in the study of approximate (null-)
  controllability of {M}arkov switch processes.
\newblock {\em Math. Control Signals Systems}, 27(4):551--578, 2015.

\bibitem{Hautus}
M.~L.~J. Hautus.
\newblock Controllability and observability conditions of linear autonomous
  systems.
\newblock {\em Nederl. Akad. Wetensch. Proc. Ser. A 72 Indag. Math.},
  31:443--448, 1969.

\bibitem{Ikeda_Watanabe_1981}
N.~Ikeda and S.~Watanabe.
\newblock {\em Stochastic Differential Equations and Diffusion Processes},
  volume~24 of {\em North-Holland Mathematical Library}.
\newblock North-Holland Publishing Co., Amsterdam--New York; Kodansha, Ltd.,
  Tokyo, 1981.

\bibitem{Jacob_Partington_2006}
B.~Jacob and J.~R. Partington.
\newblock On controllability of diagonal systems with one-dimensional input
  space.
\newblock {\em {Systems \& Control Letters}}, 55(4):321 -- 328, 2006.

\bibitem{Kalman_1959}
R.~Kalman.
\newblock On the general theory of control systems.
\newblock {\em {IRE} Transactions on Automatic Control}, 4(3):110--110, dec
  1959.

\bibitem{LiSunXiong2017}
X.~Li, J.~Sun, and J.~Xiong.
\newblock Linear quadratic optimal control problems for mean-field backward
  stochastic differential equations.
\newblock {\em Applied Mathematics {\&} Optimization}, Dec 2017.

\bibitem{LimZhou2001}
A.~E.~B. Lim and X.~Y. Zhou.
\newblock Linear-quadratic control of backward stochastic differential
  equations.
\newblock {\em SIAM Journal on Control and Optimization}, 40(2):450--474, 2001.

\bibitem{LQiYongJiongminZhangXu_2012}
Q.~L{\"u}, J.~Yong, and X.~Zhang.
\newblock Representation of it{\^o} integrals by lebesgue/bochner integrals.
\newblock {\em J. Eur. Math. Soc.}, 14:1795--1823, 2012.

\bibitem{Merton_76}
R.~C. Merton.
\newblock Option pricing when underlying stock returns are discontinuous.
\newblock {\em J. Financ. Econ.}, 3:125-- 144, 1976.

\bibitem{Pardoux_Peng_90}
E.~Pardoux and S.~G. Peng.
\newblock Adapted solution of a backward stochastic differential equation.
\newblock {\em {Syst. Control Lett.}}, 14(1):55 -- 61, 1990.

\bibitem{Peng_94}
S.~Peng.
\newblock Backward stochastic differential equation and exact controllability
  of stochastic control systems.
\newblock {\em Progr. Natur. Sci. (English Ed.)}, 4:274--284, 1994.

\bibitem{russell_Weiss_1994}
D.~Russell and G.~Weiss.
\newblock A general necessary condition for exact observability.
\newblock {\em SIAM Journal on Control and Optimization}, 32(1):1--23, 1994.

\bibitem{Schmidt_Stern_80}
E.~J. P.~G. Schmidt and R.~J. Stern.
\newblock Invariance theory for infinite dimensional linear control systems.
\newblock {\em {Applied Mathematics and Optimization}}, {6}({2}):{113--122},
  {1980}.

\bibitem{Sarbu_Tessitore_2001}
M.~Sirbu and G.~Tessitore.
\newblock {Null controllability of an infinite dimensional SDE with state- and
  control-dependent noise}.
\newblock {\em {Systems \& Control Letters}}, {44}({5}):{385--394}, {DEC 14}
  {2001}.

\bibitem{WangYangYongYu2016}
Y.~Wang, D.~Yang, J.~Yong, and Z.~Yu.
\newblock Exact controllability of linear stochastic differential equations and
  related problems.
\newblock {\em Mathematical Control and Related Fields}, 7(2):305--345, 2017.

\bibitem{yong_zhou_99}
J.~Yong and X.Y. Zhou.
\newblock {\em Stochastic Controls. Hamiltonian Systems and HJB Equations}.
\newblock Springer-Verlag, New York, 1999.

\end{thebibliography}

\end{document}